\documentclass[a4paper,french]{article}
\usepackage[T1]{fontenc}
\usepackage{babel}
\usepackage{fancyhdr}
\usepackage{geometry}
\usepackage{bibtopic}

\usepackage{graphicx}
\usepackage{xcolor}
\usepackage{pstricks}
\usepackage{subfig}
\usepackage{titletoc}

\makeatletter
\renewcommand{\subsection}{\@startsection{subsection}{2}{\z@}%
             {4ex \@plus 1ex \@minus.2ex}%
             {1.5ex \@plus.1ex \@minus.5ex}%
             {\normalfont\large\bfseries}}
\renewcommand{\subsubsection}{\@startsection{subsubsection}{3}{\z@}%
             {1.5ex \@plus .5ex \@minus.2ex}%
             {-1.5ex}%
             {\normalfont\normalsize\bfseries}}
\renewcommand{\tableofcontents}{\par\@starttoc{toc}}
\renewcommand*\l@subsection{\@dottedtocline{2}{4em}{2em}}
\let\@@seccntformat\@seccntformat
\renewcommand*{\@seccntformat}[1]{%
  \expandafter\ifx\csname @seccntformat@#1\endcsname\relax
    \expandafter\@@seccntformat
  \else
    \expandafter
      \csname @seccntformat@#1\expandafter\endcsname
  \fi
    {#1}%
}
\newcommand*{\@seccntformat@section}[1]{%
  \csname the#1\endcsname.\quad\ignorespaces
}
\newcommand*{\@seccntformat@subsection}[1]{%
  \csname the#1\endcsname.\quad\ignorespaces}
\newcommand*{\@seccntformat@subsubsection}[1]{%
  \csname the#1\endcsname.\enspace\ignorespaces}

\newcommand{\espace}{\vspace{0.4em \@plus 0.1em \@minus 0.2em}}
\makeatother

\usepackage{amsmath}
\usepackage{amssymb, amsthm}
\usepackage{xypic}
\newtheorem{thm}{Theorem}[section]  
\newtheorem{prop}[thm]{Proposition}
\newtheorem{lemma}[thm]{Lemma}
{\theoremstyle{definition}
       
         }

 \newcommand{\BBB}{\mathbf{B}} \newcommand{\CCC}{\mathbf{C}} \newcommand{\DDD}{\mathbf{D}}
 \newcommand{\FFF}{\mathbf{F}}  \newcommand{\HHH}{\mathbf{H}}
 \newcommand{\JJJ}{\mathbf{J}} \newcommand{\KKK}{\mathbf{K}} 
   \newcommand{\PPP}{\mathbf{P}}
 \newcommand{\RRR}{\mathbf{R}} \newcommand{\SSS}{\mathbf{S}} 
    
 \newcommand{\ZZZ}{\mathbf{Z}}
\newcommand{\Acc}{\mathcal{A}}   
\newcommand{\Ecc}{\mathcal{E}} \newcommand{\Fcc}{\mathcal{F}}   
\newcommand{\Icc}{\mathcal{I}}   %
\newcommand{\Mcc}{\mathcal{M}} \newcommand{\Ncc}{\mathcal{N}}  
\newcommand{\Pcc}{\mathcal{P}}   \newcommand{\Scc}{\mathcal{S}} 
 \newcommand{\Ucc}{\mathcal{P}}

\newcommand{\Mchat}{\widehat{\Mcc}}

\newcommand{\Lambdatilde}{\widetilde{\Lambda}}

\newcommand{\dbar}{\overline{\partial}}  \newcommand{\wbar}{\overline{w}} 
\newcommand{\xbar}{\overline{x}} \newcommand{\ybar}{\overline{y}} \newcommand{\zbar}{\overline{z}}

\newcommand{\Mbar}{\overline{M}}   

 \newcommand{\loc}{\mbox{\footnotesize{loc}}}     
      
\newcommand{\imm}{\mbox{Im} \,}    \newcommand{\area}{\mbox{area}}   
  \newcommand{\crit}{\mbox{Crit}}   
\newcommand{\id}{\textmd{Id}}

\newcommand{\ind}{\mbox{Ind}\, }     
               
\newcommand{\diam}{\mbox{diam}\, } 

\newcommand{\ree}{\mbox{Re}\,}       \newcommand{\gr}{\mbox{Graph}\,}

\begin{document}
\bibliographystyle{alpha}
\title{A survey on the Theorem of Chekhanov}
\author{Benoit Tonnelier}
\maketitle
\vspace{4\baselineskip}
\section*{Introduction.}\label{Section:Introduction}
Introduced by H. Hofer \cite{HoferZehnder}, the \textit{displacement energy} of a subset $X$ of the symplectic manifold $(M,\omega)$ is the minimal mean oscillation norm for a Hamiltonian to displace $X$ (see the reminders below). Floer homologies have been developped since the early work of A. Floer \cite{Floer88,Floer88bis,Floer88ter}. Estimating the displacement energy appears unquestionably as one of its important applications. In that direction, compact Lagrangian submanifolds $L$ have positive displacement energies, under natural assumptions on the symplectic topology of $(M,\omega)$ at infinity.

More precisely, the displacement energy of $L$ is greater than or equal to the minimal symplectic area of holomorphic disks bounded by $L$. This precise estimate was obtained by Y.V. Chekhanov~\cite{Chek98} in 1998\footnote{Chekhanov's paper is concerned with rational closed Lagrangian submanifolds in compact symplectic manifolds. But his work extends to the more general situation stated here.}. There exist different approaches to get it. The work presented below puts side by side those using methods related to the Lagrangian Floer homology. The estimate can be deduced from the celebrated paper of M. Gromov \cite{Gromov85}.
\subsubsection*{Acknowledgments.} I warmly thank C. Viterbo for many usefull discussions on the localization of Floer homologies. His comments on a first draft helped me to improve this paper.

\subsection*{Notations.}\label{Subsection:Notations} As usual, $\omega$ denotes a symplectic form on the manifold $M$ and $n$ is half the dimension of $M$. For an introduction to the symplectic
topology, see the classical books \cite{McDSal2,HoferZehnder}. A \textit{Hamiltonian} is a compactly supported time-depending $C^3$-function $H:[0,1]\times M\rightarrow \RRR$. Its \textit{Hamiltonian vector field} $\{X_t\}$ is implicitely defined via the formula $\iota(X_t)\omega=-dH_t$. Its flow $\{\varphi_t^H\}$ is called the \textit{Hamiltonian flow of $H$}. The \textit{displacement energy} of a compact subset $K\subset M$ is defined as 
\[E(K)=\inf\left\{\|H\|,\, \varphi_1^H(K)\cap K=\emptyset\right\}\, ,\]
where $\|H\|$ is the \textit{mean oscillation}\footnote{Often called the \textit{Hofer norm} of $H$, see for instance \cite{HoferZehnder,McDSal2}.} of $H$:
\[\|H\|=\int_0^1\left[\max H_t-\min H_t\right]dt\, .\]
Define also the functionals $\|\cdot\|_+$ and $\|\cdot\|_-$ by
\begin{align*}
\|H\|_+ & =\int_0^1(\max H_t)dt\\
\|H\|_- & =-\int_0^1(\min H_t)dt = \|-H\|_+
\end{align*}
Whenever $K\cap \varphi_1^H(K)=\emptyset$, the Hamiltonian $H$ is said to displace $K$.

An almost complex structure $J$ is $\omega$-tame when $\omega(X,JX)$ is positive for all nonzero vectors~$X$. Basic facts on the holomorphic curves are quickly recalled in subsection~\ref{Subsection:HolomorphicCurves} and appendix~\ref{Subsection:EstimatesEnergies}. Throughout the paper, the following notations are used:
\begin{align} \DDD & =\left\{z\in \CCC,\, |z|\leq 1\right\}\, ,\\
\DDD_{\pm} & = \left\{z\in \CCC,\, |z|\leq 1\mbox{ and } \pm \ree z\geq 0\right\}\, ,\\
\BBB & =\left\{z\in \CCC,\, 0\leq \imm z\leq 1\right\}\, ,\\
\BBB_R & =\left\{z\in \CCC, \, 0\leq \imm z\leq 1\mbox{ and } |\ree z|\leq R\right\}\, .
\end{align}
A compact submanifold $L$ is called \textit{Lagrangian} if $L$ is $n$-dimensional and $\omega$ vanishes along $TL$. Given an $\omega$-tame almost complex structure $J$, set $\hbar_L(J)$ to be the minimal symplectic area of non-constant $J$-holomorphic disks bounded by $L$. Set $\hbar_L=\sup\hbar_L(J)$, the supremum being taken on the space\footnote{Other assumptions on the regularity are possible. But different choices do not affect the constant $\hbar_L$.} $\Icc^3(\omega)$ of $\omega$-tame almost complex structures of class $C^3$. Provided that the symplectic manifold $(M,\omega)$ is geometrically bounded (see the definition in section \ref{Section:Gromov'sProof}), $\hbar_L>0$.
\begin{thm}[Chekanov 1998 \cite{Chek98,Oh}]\label{thm:Chekhanov}
Let $L$ be a compact Lagrangian submanifold of a geometrically bounded symplectic manifold $(M,\omega)$. Then, its displacement energy is positive:
\[E(L)\geq \hbar_L\, .\]
Moreover, for each generic Hamiltonian $H$ whose mean oscillation is less than $\hbar_L$, the intersection $L\cap \varphi_1^H(L)$ is finite, and
\[\sharp L\cap \varphi_1^H(L)\geq \sum_{i=0}^n \dim H_i(L,\FFF_2)\, .\]
\end{thm}

The above estimate is optimal, as shown by the following example. Take $L=\SSS^1\times a\SSS^1$ with $a>1$ in the symplectic vector space $\CCC^2$. Chekhanov's theorem implies that its displacement energy is exactly $\pi a^2$. Indeed the holomorphic disks (for the standard complex structure) bounded by $L$ are exactly the maps $z\mapsto (e^{i\theta}z^k, ae^{i\theta'}z^l)$ with $k,l\geq 0$.
\subsection*{Historical comments and contents.}\label{Subsection:HistoricalComments}
In the end of the sixties, V.I. Arnold \cite{Arnold} conjectured that, in a compact symplectic manifold, an exact Lagrangian submanifold $L$ is non-displaceable. In other words, its displacement energy is
infinite, which can be viewed now as a direct consequence of Chekhanov's theorem. In~1985, M. Gromov \cite{Gromov85} answered positively to the question of Arnold for compact and weakly exact symplectic manifolds. Based on Gromov's arguments, L. Polterovich \cite{Polterovich} proved in~1993 that the displacement energy of a rational compact Lagrangian submanifold is greater than or equal to~$a$ where $\omega\pi_2(M,L)=2a\ZZZ$ ($0<a<\infty$). Section~\ref{Section:Gromov'sProof} presents a weak improvement of Gromov's proof, which leads to the first part of Chekhanov's theorem.

A. Floer \cite{Floer88,Floer88bis,Floer88ter} presented a rereading of Gromov's work mixing with homological methods. He introduced the Lagrangian Floer homology, and got the above estimates on the number of intersections in the exact case. The generalisation to any Lagrangian submanifolds requires the vanishing of obstructions defined recursively and taking account the presence of holomorphic disks with non-positive Maslov indices (see \cite{FOOO}). Moreover, if it is well-defined, the Lagrangian Floer homology is zero when $L$ is displaceable. Thus, it seems not to reflect the persistence of Lagrangian intersections under small perturbations, stated by Chekhanov's theorem. This persistence can easily be checked for $C^2$-small Hamiltonians as a direct consequence of Weinstein's neighborhood theorem (see \cite{Weinstein71,Weinstein}).

In 1998, Chekhanov \cite{Chek98,Oh} defined a filtered version of the Lagrangian Floer homology, denoted here $HF_*^{(a,b]}(L,\omega;H,J_0)$. It is perfectly well-defined for a compact Lagrangian submanifold~$L$ provided that $b-a<\hbar_L(J_0)$ for a fixed $\omega$-tame and geometrically bounded almost complex structure $J_0$. Following Chekhanov, when $\|H\|_+\leq b$ and $\|H\|_-<-a$, the continuation maps give:
\[HM_*(L)\rightarrow HF_*^{(a,b]}(L;H,\omega,J_0)\rightarrow HM_*(L)\, ,\]
whose composition is the identity (proposition~\ref{prop:2.4}). The rank-nullity theorem leads to the estimates stated in Theorem \ref{thm:Chekhanov}. Those maps can be defined as a local version of PSS maps, which were first introduced in \cite{PSS} (Piunikhin-Salamon-Schwarz). This different approach, presented in section~\ref{Section:Kerman'sProof}, is due to Kerman \cite{Kerman1} at least for a local version of the Hamiltonian Floer homology.\espace

There exists a slight different approach of displaceability, based on the action selectors \cite{Viterbo,Viterbo2,Hermann}. We de not evoke it within the present paper, which presents in details the three aproaches mentionned above. Here is the \textit{table of contents}.\setcounter{tocdepth}{2} 
\tableofcontents \espace Throughout the paper, gluing and splitting are not explicitly justified. Most transversality arguments are skipped. The reader is referred to \cite{McDSal} for details.
\newpage

\section{Gromov's proof}\label{Section:Gromov'sProof}
This section is devoted to revisiting the celebrated paper \cite{Gromov85} of M. Gromov. Assume $(M,\omega)$ to be geometrically bounded (\cite{AudLaf}, Chapter V, definition 2.2.1):\espace

\textbf{C1 --} There is a Riemannian metric $g$, such that, for some positive constant $C$, the injectivity radius of $g$ is greater than $3/C$ and the sectional curvature of $g$ is less than $C$.\espace

\textbf{C2 --} There exists a smooth almost complex structure $J_0$ such that, for every tangent vectors~$X$, we have: $C\omega(X,J_0X)\geq \|X\|^2$ and $|\omega(X,Y)|\leq C\|X\|.\|Y\|$.\espace

Examples of geometrically bounded symplectic manifolds include symplectic vector spaces, closed symplectic manifolds and cotangent bundles of compact manifolds. More general examples are constructing by adding cones to compact symplectic manifolds with contact type boundaries. The above conditions were already stated in the original work of Gromov. In 1985, Gromov \cite{Gromov85} introduced the holomorphic curves in symplectic topology. He proved that, in a weakly exact\footnote{Weakly exact = The symplectic form $\omega$ vanishes on $\pi_2(M)$.} and geometrically bounded symplectic manifold, a displaceable compact Lagrangian submanifold~$L$ must bound at least one holomorphic disk (\cite{Gromov85}, Section 2.3). The arguments developped there serve to prove
\begin{thm}\label{Thm:Gromov}
Let $L$ be a compact Lagrangian submanifold of a geometrically bounded symplectic manifold $(M,\omega)$. Then, its displacement energy is positive: \[E(L)\geq \hbar_L\, .\]
\end{thm}
The constant $\hbar_L$ is defined in subsubsection \ref{Subsubsection:ConstanthbarL}. From (\cite{Gromov85}, 2.3-B), the existence of non-constant holomorphic curves is guaranted by the non-existence of solutions of some elliptic equations, as those studied in subsection \ref{Subsection:FloerContinuationStrips}. The proof given in subsection \ref{Subsection:Gromov'sProofRevisited} simply adds estimates on the energies.
\subsection{Reminders on holomorphic curves.}\label{Subsection:HolomorphicCurves}
In the sequel, we will need to perturb $J_0$. Fix $A>C$. Let $\Icc^3_A(J_0)$ stand for the space of almost complex structures $J$ of class $C^3$,

- equal to $J_0$ outside a sufficiently large compact subset of $M$ ;

- and satisfying: $A\omega(X,JX)\geq \|X\|^2$ for all tangent vectors $X$.

\noindent Note that $\Icc^3_A(J_0)$ is a smooth Frechet manifold.
\subsubsection{\!\!\!\!} A compact Riemannian surface $\Sigma$ can be viewed as a compact, orientable real surface, equipped with a complex structure $j$. Given a $\Sigma$-parametrized family\footnote{Id est, a map $\Sigma\rightarrow \Icc_A^3(J_0)$ of class $C^3$.} $\JJJ=\{J_z,\, z\in \Sigma\}$ of almost complex structures in $\Icc^3_A(J_0)$, a \textit{$\JJJ$-holomorphic curve} is a  map $u:\Sigma\rightarrow M$ of class $W^{1,p}$ (with $p>2$) satisfying the {\em Cauchy-Riemann equation} (\cite{McDSal}, section 2.2)
\begin{align}
\dbar_Ju(z)=\frac{1}{2}\left[du(z)+J_z\circ du(z)\circ j\right]=0\,
.\label{eqn:CR}\end{align}
For an introduction to holomorphic curves, see \cite{ABKLR,AudLaf,Humm97,McDSal}. The {\em energy} of $u$ is defined as
\[E(u)=\int_{\Sigma}u^*\omega>0\, .\]

Condition C2 implies that $u^*\omega$ does not vanish on $\Sigma$ (\cite{ABKLR}, section 6.3.2). Upper bounds on the energy yield estimates on the diameter of the holomorphic curve (\cite{AudLaf}, Chapter V, proposition~4.4.1). 
\begin{prop}\label{prop:BoundsEnergy}
There exists a constant $c_0$ independent from $M$, $C$ or $A$, such that the following holds. With the above notations, for a connected $\JJJ$-holomorphic curve $u:\Sigma\rightarrow M$ (possibly with non-empty boundary), we have:
\begin{align} \diam \left[u(\Sigma)\right]\leq \frac{2}{C} \left(1+\frac{AC^2}{c_0} E(u)\right)\, .\label{Eqn:EstimateDiameter}\end{align}
\end{prop}
\noindent Its proof is postponed to Appendix \ref{Subsubsection:EstimatesEnergies0}.
\subsubsection{\!\!\!\!}\label{Subsubsection:ConstanthbarL} For any $J\in \Icc^3_A(J_0)$, a $J$-holomorphic curve is at least of class $C^3$. For $\alpha\in \pi_2(M)$, set
\[\Scc(\alpha,J)=\left\{u:\SSS^2\rightarrow M,\, \dbar_Ju=0,\, [u]=\alpha\right\}\, .\]
For a generic choice of $J$, the space $\Scc(\alpha,J)$ is a submanifold of $W^{1,p}(\SSS^2,M)$ (with $p>2$) of dimension $2n+2c_1(\alpha)$ (\cite{McDSal2}, chapter~3). Here, $c_1$ denotes the Chern class associated to $J$, but depends only on $\omega$. The symplectic manifold $(M,\omega)$ is said to be {\em semipositive} (\cite{McDSal}, subsection~6.4) when $3-n\leq c_1(\alpha)\leq 0$ implies $\omega(\alpha)\leq 0$.
\begin{lemma} Let $L$ be a compact Lagrangian submanifold of $(M,\omega)$. Then, the infinimum $\hbar_L(J)$ of the energies of $J$-holomorphic disks bounded by $L$ is positive.\end{lemma} 
\begin{proof}
The tubular neighborhood theorem asserts that $V_r(L)=\left\{x\in M,\, d(x,L)\leq r\right\}$ contracts onto~$L$ for $r>0$ sufficiently small. Let $u:(\DDD,\partial \DDD)\rightarrow (M,L)$ be a non-constant $J$-holomorphic disk bounded by $L$. The image of $u$ cannot be contained in $V_r(L)$ as its energy is $\omega(u)\neq 0$. Thus, there exists $z\in \DDD$ such that $u(z)\in \partial V_{r}(L)$. Applying Proposition \ref{prop:BoundsEnergy} gives
\[E(u)\geq \frac{c_0}{2AC^2}\left(r- \frac{2}{C } \right) \, .\]
Here, the constant $C$ can be fixed sufficiently large so that $Cr>2$, as $c_0$ is independent from $C$. The lemma is established.\end{proof}
As a consequence, the constant $\hbar_L$ appearing in theorems \ref{thm:Chekhanov} and \ref{Thm:Gromov} is positive. Moreover, the map $J\mapsto \hbar_L(J)$ is lower semi-continuous.
\subsection{Floer continuation strips.}\label{Subsection:FloerContinuationStrips}
\subsubsection{\!\!\!\!} An $L$-connecting path is a continuous map $x:[0,1]\rightarrow M$ satisfying the boundary conditions $x(0),x(1)\in L$. It is said to be contractible when $[x]=0\in \pi_1(M,L)$. When $x$ is a trajectory of $X_H$ for a Hamiltonian $H:M\times [0,1]\rightarrow \RRR$, the path $x$ is called an {\em $L$-orbit} of $H$. Proving the persistence of intersections under the Hamiltonian flow of $H$ amounts to detecting $L$-orbits. The proof given in subsection \ref{Subsection:Gromov'sProofRevisited} requires a Hamiltonian perturbation on the Cauchy-Riemann equation (\ref{eqn:CR}).
Let $H_{\pm}:\SSS^1\times M\rightarrow M$ be two Hamiltonians. Given a compact homotopy\footnote{Following Kerman \cite{Kerman1,Kerman2},
a compact homotopy is a $\RRR$-parameterized path in a functional space, locally constant at infinity. Here, the functional space is $C^{3}_c(M,\RRR_+)\times \Icc_A^3(J_0)$. Note that the union of the supports of the different Hamiltonians $H_s$ is compact.} $(H_s,J_s)$ from $(H_-,J_-)$ to $(H_+,J_+)$, a {\em Floer continuation strip} $u:\BBB\rightarrow M$ is a solution with finite energy of the
Floer equation (\cite{HoferSalamon}, equation (2))
\begin{align}\partial_su+J_{s,t}(u)\left[\partial_tu-X_{s,t}(u)\right]=0\, ,\label{eqn:Floer}\end{align}
with the boundary conditions
\[u(\partial\BBB)\subset L\, .\]
Here, $X_{s,t}$ denotes the Hamiltonian vector field associated to $H_{s,t}$. The energy of $u$ is defined as 
\[E(u)=\int_{-\infty}^{+\infty}\int_0^1{|\partial_su|}^2\, dtds\, .\]
\begin{prop}\label{prop:LocalisationFloerStrips}
Fix $L$, $J_0$, $\{H_s\}$ and $\{J_s\}$ as above. Assume that the compact homotopy $\{J_s\}$ lies in $\Icc^2_A(J_0)$. For a Floer continuation strip $u$,
\[u(\BBB)\subset V_R(K)\quad \mbox{ where } \quad R=\frac{2}{C}\left[1+\frac{AC^2}{c_0}E(u)\right]\, ,\]
and the compact $K$ is the union of $L$ and the supports of all the Hamiltonians $H_s$.
\end{prop}
\noindent We have set 
\[V_R(K)=\left\{x\in M,\, \exists y\in K,\, d_M(x,y)\leq R\right\}\, .\]
\begin{proof}
Each connected component $U$ of $u^{-1}(M-K)$ is the increasing union of connected regular open subsets $\Sigma_n$ (i.e., with smooth boundaries). The restriction of $u$ to $\Sigma_n$ is simply a $\{J_{t}\}$-holomorphic curve $v_n$ and $v_n^*\omega={\left|\partial_sv_n\right|}^2ds\wedge dt$. Thus, its energy is less than $E(u)$. Proposition~\ref{prop:BoundsEnergy} gives:
\[d(u(z),u(\partial \Sigma_n))\leq R\]
for all $z\in \Sigma_n$. Thus, $d(u(z),K)\leq R+\epsilon$ for $z\in U$ where $\epsilon>0$ is as small as we want.\end{proof}
\subsubsection{Limits at $\pm\infty$.}\label{Subsubsection:Limits} Let $u$ be a Floer continuation strip for the compact homotopy $(H_s,J_s)$. As $\{H_s\}$ goes from~$0$ to~$0$, there exists $S>0$ such that $H_{\pm s}=0$ and $J_{\pm s}=J_{\pm}$ for~$s>S$. Let $\{\varphi_t^+\}$ be the Hamiltonian isotopy defined by $H_+$, and set $u(s,t)=\varphi_t^+\circ v_+(s,t)$. Then, $v_+$ is $J_+$-holomorphic on the half-band $(S,\infty)\times [0,1]$.

Let us consider the curves $v_s:t\mapsto v(s,t)$. From (\cite{McDSal}, lemma 4.3.1), there exists a constant $C>0$ such that
\[\mbox{length } (v_s)\leq C E(v|(s-1,\infty)\times [0,1])\, .\]
From a straightforward computation, the energy of $v$ on $(s-1,\infty)\times [0,1]$ equals the energy of $u$ on the same domain, for $s>S+1$. Thus, the lengths of $v_s$ go to zero when $s\rightarrow \infty$. Moreover, they take values inside a compact subset of $M$ (proposition~\ref{prop:LocalisationFloerStrips}). From Arzela-Ascoli's theorem, there exists a subsequence $\left(v_{s_n}\right)$ converging to a constant curve $x$. Necessarly, $x$ belongs to $L$ as a limit of $u(s_n,0)$. It then follows that $\left(u_{s_n}:t\mapsto u(s_n,t)\right)$ converges to the curve $t\mapsto \varphi_t(x)$, which is an $L$-orbit of $H_+$. See also (\cite{Salamon}, proposition 1.21).
\begin{prop}\label{prop:Limits} 
Assume the Hamiltonian $H_+$ to displace $L$. Then there is no Floer continuation strip for any compact homotopy $(H_s,J_s)$, where $\{H_s\}$ ends at $H_+$.
\end{prop}
Sometimes, the curves $(s\mapsto u(s,t))$ converge to $L$-orbits $x_{\pm}$ of $H_{\pm}$ when $s\rightarrow \pm\infty$. (This is the case when $H_+$ and $H_-$ meets generic conditions, to be stated in subsubsection \ref{Subsubsection:ConleyZehnder}.) The strip $u$ is said to \textit{go from $x_-$ to $x_+$.} (See \cite{HoferSalamon,McDSal,Salamon} for details.)\espace

From (\cite{Gromov85}, section 2.3.B), the displaceability of $L$ implies the existence of a non-constant holomorphic disk $u$. In particular, $\omega(u)$ is positive, and then $L$ is not exact. The proof below can be seen a refinement of this argument.

\subsection{Proof of theorem \ref{Thm:Gromov}.}\label{Subsection:Gromov'sProofRevisited}
Fix a Hamiltonian $H:M\times [0,1]\rightarrow M$ which displaces $L$. Take $J_0$ so that $\hbar_L<\hbar_L(J_0)+\epsilon$ where $\epsilon>0$ is arbitrarily small. Fix a non-decreasing smooth function $\beta:\RRR\rightarrow [0,1]$ equal to~$0$ for~$s<-1$ and such that $\beta(s)+\beta(-s)=1$. Consider the compact homotopies from~$0$ to~$0$:
\[H_{s,t}^R=\beta(s+R)\beta(-s+R)H_t\]
with $R\in \RRR$, and denote by $X_{s,t}^R$ the associated Hamiltonian vector field. Complete it by a one-parameter family $\JJJ=\{J^R_{s,t},\, R\in \RRR\}$ of compact homotopies from $J_0$ to $J_0$. Assume $J^R_{s,t}=J_0$ for $R<-1$ and $J^r_{s,t}=J_t$ for $|s|<R-1$ and $|s|>R+1$. Here, $J_{s,t}^R\in \Icc_A^2(J_0)$ with $A>C$ sufficiently large. Moreover, assume\footnote{Where $J_z^R=J_{s,t}^R$ for $z=s+it$.} $\hbar_L<\hbar_L(J_z^R)+\epsilon$ for all $z$ and $R$.\espace
\begin{figure}
\centering
\includegraphics[height=65mm]{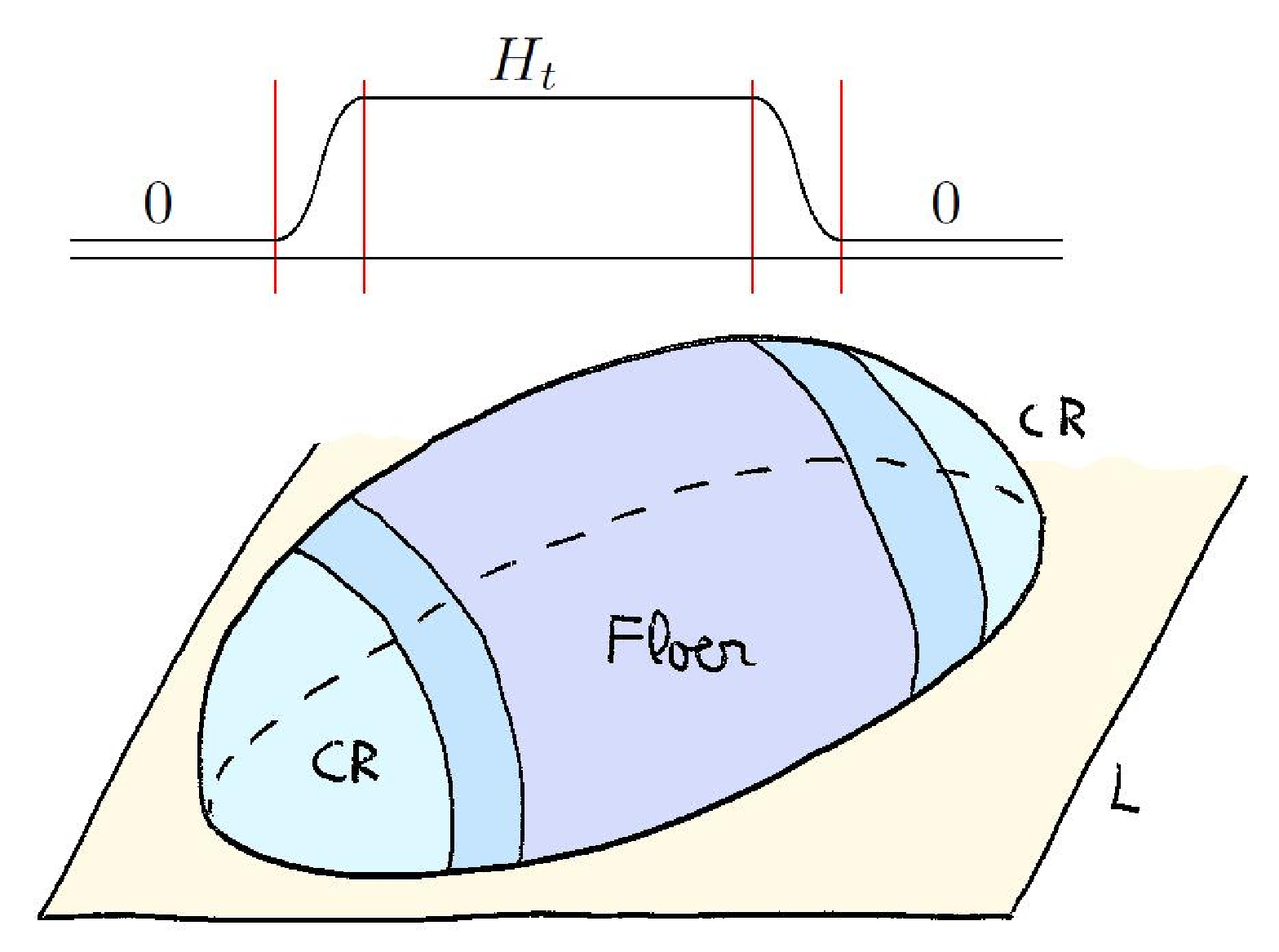}
\caption{An element of $\Ncc_L(\HHH,\JJJ)$}\label{Fig:ElementNcLHJ}
\end{figure}
For submanifolds $X$ and $Y$ of $L$, we introduce the following spaces
\begin{align}
\Ncc_L(\HHH,\JJJ) & =\left\{R,u:\BBB\rightarrow M,\,\mbox{ st } \begin{matrix}
\partial_su+J_{s,t}^r(u)\left[\partial_tu-X_{s,t}^R(u)\right]=0\\ \lim_{s\rightarrow -\infty}u(s,t)=x\\ \lim_{s\rightarrow +\infty}u(s,t)=y\\ [u]=0\in\pi_2(M,L)
\end{matrix}\right\}\, ,\label{eqn:NcHJ}\\ 
\Ncc_L(X,Y;\HHH,\JJJ) & = \left\{(R,u,x,y)\in \Ncc(H,\JJJ),\, \mbox{ st } x\in X, \, y\in Y\right\}\, ,\label{eqn:NcXYHJ} \end{align}
endowed with the topology of uniform $C^{2}$-convergence on compact subsets of $\BBB$.  For a pair $(R,u)$ in $\Ncc_L(\HHH,\JJJ)$, Existence of limits at $\pm\infty$ guarantee $E(u)<\infty$ (proposition \ref{prop:Limits}). Lemma \ref{lemma:A3} gives:
\begin{align} E(u)\leq \|H\|\, . \label{eqn:EuleqH} \end{align}
\subsubsection{Proof.}\label{Subsubsection:ProofGromovOutline} The linearization at $(R,u)$ of equation appearing in (\ref{eqn:NcHJ}) defines a Fredholm linear map $D_u:\RRR\oplus W^{1,p}(u^*TM)\rightarrow L^p(u^*TM)$. More precisely, it is the sum of the Cauchy-Riemann operator and a compact operator depending on $X^R_{s,t}$. As~$u$ is homotopic rel~$L$ to a constant disk, its index is necessarly $n+1$. For a precise computation, see \cite{RobbinSalamon1,RobbinSalamon2} or (\cite{Gromov85}, section~2.1). Thus, the expected dimensions are:
\begin{align}
\dim \Ncc_L(\HHH,\JJJ) & =n+1\, , \label{eqn:DimNLHJ}\\
\dim \Ncc_L(X,Y;\HHH,\JJJ) & = 1+\dim X+\dim Y-n\, .\label{eqn:DimNLXYHJ}
\end{align}
The open subset $\{R<-1\}$ of $\Ncc_L(X,Y;\HHH,\JJJ)$ is readily $X\cap Y\times (-\infty,-1)$. The submanifolds $X$ and $Y$ of $L$ are assumed to have transverse intersections.
\begin{lemma}\label{Lemma:SmoothMnflds} Let $H$ be an Hamiltonian of class $C^l$. Assume $\JJJ$ to be a generic family of class $C^l$. 
\begin{itemize}
\item If $l\geq n+3$, then  the space $\Ncc_L(\HHH,\JJJ)$ is in a natural way a manifold with the expected dimensions (\ref{eqn:DimNLHJ}) ;
\item If $l\geq 3$, then the space $\Ncc_L(X,Y;\HHH,\JJJ)$ is a manifold of dimension (\ref{eqn:DimNLXYHJ}), for a generic choice of $\JJJ$.
\end{itemize}
\end{lemma}
The proof of the above lemma is postponed to the next subsection.\espace
\begin{figure}
\centering
\begin{picture}(20,11)
\put(0,1){\vector(1,0){20}} \put(19,0){$\RRR$}
{\red \put(0,6){\line(1,0){7}} \put(3,6.2){$S$} \put(2,6){\vector(1,0){1.5}}
\qbezier(7,6)(8,6)(9.5,7) \qbezier(9.5,7)(11,8)(12,7.5)
\qbezier(12.3,7.35)(15,6)(14,4) \qbezier(14,4)(12,0)(19,3) \put(18,3.5){Limit ?}} {\blue \put(12,9){A compact component, ignore it.}
\qbezier(12,7.5)(13,8.5)(14,8.5) \qbezier(14,8.5)(23,8.5)(14.2,4.6)
\qbezier(13.8,4.4)(12.4,3.7)(11.1,4.2)
\qbezier(10.8,4.3)(9.5,4.8)(9.5,3) \qbezier(9.5,3)(9.5,1.4)(10,1.4)
\qbezier(10,1.4)(10.5,1.4)(11,4.2) \qbezier(11,4.2)(11.5,7)(12,7.5)}
{\put(7,0.5){\line(0,1){10.5}} \put(6,0){$-1$}}
\end{picture}
\caption{A formal représentation of $\Ncc_L(x,L;\HHH,\JJJ)$}\label{fig:RepresentationNcxL}
\end{figure}

Assume $X=\{x\}$ and $Y=L$. Then, $\Ncc(x,L;\HHH,\JJJ)$ is a one-dimensional manifold. The first projection $(R,u,x,y)\mapsto R$ is at least continuous. The open subset $\{R<-1\}$ is the set $(-\infty,-1)\times \{x\}$ where $x$ is viewed as the constant disk equal to $x$. Set $S$ for its connected component. See figure \ref{fig:RepresentationNcxL}. Following $S$ (the red line), one gets a non-convergent sequence $(R_n,u_n)$ of $S$. But estimate (\ref{eqn:EuleqH}) and proposition~\ref{prop:BoundsEnergy} imply that the images of $u$ lie in the compact $V_S(K)$ where $S=\frac{2}{C}\left[1+\frac{AC^2}{c_0}\|H\|\right]$. Thus, Chapters $4$ and $12$ of \cite{McDSal2} show that the sequence $(u_n)$ admits a subsequence, still denoted by $(u_n)$, converging $C^{2}$-uniformly on compact subsets of $\BBB-F$. Here, $F$ is a finite subset of $\BBB$ where bubbling off of holomorphic spheres or disks can occur.
\begin{itemize}
\item If $R_n\rightarrow \infty$, the limit $v:(\BBB-F,\partial \BBB)\rightarrow (M,L)$ satisfies the Floer equation \begin{align}\partial_sv+J_t\left[\partial_tv-X_t\right]=0 \label{eqn:Floer eqn}\end{align} As $v$ is of finite energy, the singularities can be removed (\cite{McDSal2}, Chapter 4), and $v$ can be smoothly extended to a Floer continuation strip for the constant homotopy $(H_t,J_t)$. Once again, as $v$ is of finite energy, a subsequence of $\left(t\mapsto v(s,t)\right)$ for $s\rightarrow \infty$ admits a limit $x^+$, which is an $L$-orbit of $H$ (proposition \ref{prop:Limits}). As the Hamiltonian $H$ displaces $L$, such a Floer continuation strip $v$ cannot exist. 
\item Thus, the sequence $(R_n)$ is bounded, and we may assume $R_n\rightarrow R\geq -1$. In this case, as $(u_n)$ does not converge in $S$, there must be a bubbling off of at least one $J_z^R$-holomorphic sphere or disk, with $z\in \BBB$. Holomorphic spheres can be avoided by generic data. Up to removable singularities, this holomorphic disk $u$ is obtained as a limit of $w\mapsto u_n(z_n+\rho_n w)$ with well-chosen sequences $z_n\rightarrow z$ and $\rho_n\rightarrow 0$ (see~\cite{Salamon}). Thus, \[\hbar_L-\epsilon\leq\hbar_L(J_z^r)\leq E(u)\leq \limsup E(u_n)\leq \|H\|\, .\]
\end{itemize}
\noindent Take the infinimum on $\epsilon>0$, and afterwards, the infinimum on $H$ displacing $L$. Hence, we get:
\[\hbar_L\leq E(L)\, .\]
\subsubsection{Proof of Lemma \ref{Lemma:SmoothMnflds}.}\label{Subsection:ProofOfLemma} 
Lemma \ref{Lemma:SmoothMnflds} lies on a well-known transversality argument, given for instance in \cite{Floer88,Salamon,McDSal}. But we have to chek that  the conditions required on $\JJJ$ can be satisfied. It requires the following version of Sard's theorem, due to Smale~\cite{Smale}.
\begin{thm}[Smale~\cite{Smale}] Let $\Ncc$ and $\Ecc$ be two separable Banach manifolds. Let $\Fcc:\Ncc\rightarrow \Ecc$ be a smooth map of class $C^{k+1}$, such that all the differentials $d\Fcc(x)$ are Fredholm operators of index~$k$. Then, the non-regular values of $\Fcc$ is a set of first category. For any regular value $y$, its preimage $\Fcc^{-1}(y)$ is a $k$-dimensional submanifold of $X$.\end{thm}

For a pair~$(R,u)$ in $\Ncc_L(\HHH,\JJJ)$, inequality~(\ref{eqn:EuleqH}) and proposition~\ref{prop:LocalisationFloerStrips} imply that $u$ is contained in~$V_S(K)$ where $K$~is as in proposition \ref{prop:LocalisationFloerStrips} and $S=1+\frac{2}{C}\left[1+\frac{AC^2}{ c_0}\| H\|\right] $. Perturbations on~$\JJJ$ may be realized in~$V_S(K)$. Now, introduce the following Banach
manifolds:\espace

-- For $p>2$, the space $\Ncc_L^p$ collects the pairs $(R,u)$ where the map $u:(\BBB,\partial \BBB)\rightarrow (M,L)$ of class $W^{1,p}_{loc}$ converges to points of $L$ at $\pm\infty$ and is of class $W^{1,p}$ in their neighborhoods. In other words, we assume the existence of maps $w_{\pm}:\{z\in\CCC,\, \pm\imm z\geq 0\}\rightarrow M$ of class $W^{1,p}_{\loc}$ such that 
\begin{align*}
u(s,t) & = w_-\left(\exp(\pi z)\right)\, ,\\
 & = w_+\left(\exp (\pi-\pi z)\right)\, .\end{align*}

-- The space $\Icc^l$ collects parametrized families $\JJJ=\{J_{s,t}^R\}$ of class $C^l$ of compact homotopies from $J_0$ to $J_0$ inside $\Icc_A^l(J_0)$ and equal to $J_0$ outside $V_{S}(K)$ where $K$ is as in proposition \ref{prop:LocalisationFloerStrips}. Moreover, we require $J_{s,t}^R=J_0$ for $R<-1$ and $J_{s,t}^R=J_0$ for $|s|<R-1$ and $|s|>R+1$.\espace

-- For $u\in \Ncc_L^p$, the tangent space $T_u\Ncc_L^p$ is the space of sections of class $W^{1,p}$ of the Hermitian vector bundle $u^*TM\rightarrow \BBB$, this pullback being of class $W^{1,p}$. Let $\Ecc^p$ be the vector bundle of sections of $u^*TM\rightarrow \BBB$ of class $L^p$.\espace

From subsubsection~\ref{Subsubsection:Limits}, $\Ncc_L(\HHH,\JJJ)$ can be seen as a subset of the smooth Banach manifold $\Ncc_L^p$. Namely, it is the zero set of the global section
\[\Fcc_{\JJJ}:\begin{matrix}\RRR\times\Ncc_L^p & \rightarrow & \Ecc^p\\ (R,u) & \mapsto & \partial_su+J_{s,t}^R(u)\left[\partial_tu-X_{s,t}^R(u)\right]\, .\end{matrix}\]
For $(R,u)\in \Ncc_L(\HHH,\JJJ)$, the vertical derivative (that means, the vertical component of the differential $d\Fcc_{\JJJ}(R,u)$) is
\[D_{R,u}:\begin{matrix}\RRR\oplus W^{1,p}(\BBB,u^*TM) & \rightarrow & L^p(\BBB,u^*TM)\\ (\delta r,\delta u) & \mapsto & \partial_su+J_{s,t}^R(u)\partial_t\xi+\nabla_{\xi}J^R_{s,t}(u)\partial_t u+A(\delta r,\delta u)\, ,\end{matrix}\]
where $A$ is a compact operator, depending on $\{X_{s,t}^R\}$. It follows that $D_{R,u}$ is a Fredholm operator of index $n+1$ (\cite{McDSal}, appendix~C). If $D_{R,u}$ is onto for all $(R,u)\in \Ncc_L(\HHH,\JJJ)$, then this space is a $(n+1)$-dimensional manifold.

From (\cite{McDSal} p. 48), the vector bundle $\Ecc^p\rightarrow \RRR\times \Ncc^p_L\times \Icc^l$ is of class $C^{l-1}$, and $\Fcc(J,u)=\Fcc_{\JJJ}(u)$ defines a section of class $C^{l-1}$, provided that $H$ is of class $C^l$. The vertical derivatives of $\Fcc$ along its zero set are surjective operators. The implicit function theorem shows that the union of $\Ncc_L(\HHH,\JJJ)\times \{\JJJ\}$ where $\JJJ$ describes $\Icc^l$ is a submanifold $\Ncc_L(\HHH,\Icc^l)$ of $\RRR\times \Ncc_L^p\times \Icc^l$ of class $C^{l-1}$, see \cite{HoferSalamon} or (\cite{McDSal}, proposition 2.3.1). The next argument is based on the properties of the second projection 
\[\pi:\Ncc_L(\HHH,\Icc^l)\rightarrow \Icc^l\, .\]
This map is of class $C^{l-1}$. The tangent space of $\Ncc_L(\HHH,\Icc)$ at $(R,u)$ is given by
\[T_u\Ncc(\HHH,\Icc)=\left\{(\delta R,\delta u,\delta J),\, D_{R,u}(\delta R,\delta u)+\delta J_{s,t}^R \partial_tu=0\right\}\, .\]
The kernel of $d\pi(R,u,J)$ is exactly the kernel of $D_u$. Standard methods in functional analysis prove that all the differentials $d\pi(R,u,J)$ are Fredholm operators of index $n+1$, see (\cite{McDSal}, appendix~A).  For  $l-2\geq n+1$, Sard'-Smale's theorem \cite{Smale} implies that the regular values of~$\pi$ form a dense set of $\Icc^l$. For a regular value $\JJJ\in \Icc$, the operator $D_u$ is onto for every curve $u\in \pi^{-1}(\JJJ)=\Ncc_L(\HHH,\JJJ)$, and thus the space $\Ncc_L(\HHH,\JJJ)$ is a $(n+1)$-dimensional submanifold of $\RRR\times \Ncc_L^p$.\espace

Still with the above notations, the map
\[\mbox{ev}:\begin{matrix} \Ncc_L(\HHH,\Icc^l) & \rightarrow & L\\ u & \mapsto & \lim_{s\rightarrow -\infty} u(s,0)\end{matrix}\]
is a submersion onto $L$. Thus, $\Ncc_L(\{x\},L;\HHH,\Icc^l)=\mbox{ev}^{-1}(x)$ is a closed submanifold of codimension $n$. The projection $\pi$ restricts to a Fredholm map
\[\pi:\Ncc_L(\{x\},L;\HHH,\Icc^l)\rightarrow \Icc^l\]
of index $1$. When $l-2\geq 1$, the regular values form a dens subset of $\Icc^l$. For regular value $\JJJ\in \Icc^l$, the space $\Ncc_L(\{x\},L;\HHH,\Icc)$ is a one-dimesnional submanifold.

\newpage
\section{Chekhanov's proof.}\label{Section:Chekhanov'sProof}
This section revisits the work of Chekhanov \cite{Chek98,Oh}.
\subsection{Filtered Lagrangian Floer homology.}\label{Subsection:FilteredLagrangianFloerHomology} We set up here a {\em filtered version} of the Lagrangian Floer homology. For a Hamiltonian $H$ called \textit{admissible} with respect to $L$, we define homology groups denoted 
\begin{align} HF_*^{(a,b]}(L,\omega;H,J_0)\, , \end{align}
where the interval $(a,b]$, called the \textit{action window}, has length $b-a<\hbar_L(J_0)$. This condition removes some problems due to the presence of holomorphic disks (see \cite{FOOO}), and the definition given in subsection \ref{subsubsection:BoundaryOperator} is available for all compact Lagrangian submanifolds $L$, provided that the symplectic manifold $(M,\omega)$ is geometrically bounded. The construction requires only standard arguments dating back to the original work of A. Floer \cite{Floer88,Floer88bis,Floer88ter}.
\subsubsection{The Floer module.}\label{Subsubsection:FloerModule}
The symplectic form $\omega$ defines two morphisms $\pi_2(M,L)\rightarrow\RRR$: the \textit{symplectic action} $\omega$ and the \textit{Maslov index}\footnote{Recall its definition. Given a disk $u:(\DDD,\partial \DDD)\rightarrow (M,L)$ of class $C^1$, note $x:\SSS^1\rightarrow L$ its boundary. Then, $x^*TL$ may be viewed as a loop $\gamma$ of Lagrangian subspaces of $\CCC^n$ via a symplectic trivialization of $u^*TM$. Set: $\omega(u)=\int_{\DDD}u^*\omega$ and $\mu(u)=\mu_{RS}(\RRR^n,\gamma)$. Here, $\mu_{RS}$ is the Robbin-Salamon index for the Lagrangian paths~\cite{RobbinSalamon1,RobbinSalamon2}. This definition does not depend on $u$ up to an homotopy.} $\mu$. In the sequel, $\Lambda_L M$ stands for the space of contractible $L$-connecting paths of class $C^1$. Let $\Lambdatilde_L M\rightarrow \Lambda_L M$ be its universal covering. Basically, a point in $\Lambdatilde_LM$ is represented by a half-disk $w:(\DDD_+,\partial_0\DDD_+)\rightarrow (M,L)$ of class $C^1$ bounded by a $L$-connecting path $x$. Here, $\partial_0\DDD_+$ denotes the segment $[-1,1]$ viewed as the lower boundary of the upper unit half-disk $\DDD_+$, and $\partial_+\DDD_+$ denotes its upper bound, parametrized by $t\in [0,1]$ as~$e^{i\pi t}$. The map $w$ is called a \textit{capping half-disk} of $x$.

Introduce the Galois covering $\Lambda_L'M\rightarrow \Lambda_LM$ whose deck group is given by the quotient
\[\Gamma(\omega)=\pi_2(L,M)/\ker \omega\cap \ker I_{\mu}\, .\]
In other words, $\Lambdatilde_L'M$ is the quotient of $\Lambda_LM$ under the action of $\ker\omega\cap\ker I_{\mu}$. Two pairs $[x,w]$ and $[x,w']$ define the same point in $\Lambda_L'M$ whenever the disk $w\sharp \wbar '$ is vanished by both $\omega$ and $\mu$.\espace

For a Hamiltonian $H:M\times [0,1]\rightarrow \RRR$, the action functional $\Acc_H:\Lambda_L'M\rightarrow\RRR$ is defined by 
\[\Acc_H[x,u]=\int_0^1 H_t(x_t)dt -\int_{\DDD_+}u^*\omega\, .\]
The formal critical points of $\Acc_H$ are precisely the \textit{capping $L$-orbits} $[x,w]$ of $H$, i.e. points of $\Lambda '_L M$ above contractible $L$-orbits of the Hamiltonian flow of $H$. They form a set, denoted by~$\Pcc'_L(H,\omega)$.

The \textit{relative Floer module} $CF(L,\omega;H)$ is the $\FFF_2$-vector space generated by $\Pcc_L'(H,\omega)$, equipped with the valuation: 
\[v_H(\xi) =\sup\left\{ \Acc_H[x,w],\, \xi[x,w]\neq 0\right\}\, .\]
Set:
\begin{align*} CF^a(L\omega;H) & =\left\{\xi\in CF(L,\omega;H),\, v_H(\xi)\leq a\right\}\, ;\\ \mbox{ and } CF^{(a,b]}(L,\omega;H) & = CF^{b}(L,\omega;H)/CF^a(L,\omega;H)\, .\end{align*}
\subsubsection{The Conley-Zehnder index.}\label{Subsubsection:ConleyZehnder} As yet, no assumption was made on the Hamiltonian~$H$. An $L$-orbit $x$ is called \textit{non-degenerate} when $d\varphi_1^H(x_0)(T_{x_0}L)$ is transverse to $T_{x_0}L$. Given a capping half-disk $w:(\DDD_+,\partial_0\DDD_+)\rightarrow (M,L)$ of $x$, the {\em Conley-Zehnder index} of $\xbar=[x,w]$ is defined as follows. Let $\Phi:w^*TM\rightarrow \CCC^n$ be a symplectic trivialization of $w^*TM$. Set
\begin{align*} \forall t\in [0,1], \, \Lambda_1^{\Phi} (t) & = \left[\Phi(t)T_{u(t)}L\right]\oplus \left[\Phi(-t)T_{u(t)}L\right]\\ \Lambda_2^{\Phi}(t) & = \gr\left[\Phi(e^{i\pi t})d\varphi_t^H(x_0) \Phi(1)^{-1}\right]\, . \end{align*}
Here, $\Lambda_1^{\Phi}$ and $\Lambda_2^{\Phi}$ are paths of Lagrangian subspaces of $\CCC^n\oplus \CCC^n$. Such a pair is associated to a half-integer $\ind\left(\Lambda_1^{\Phi},\Lambda_2^{\Phi}\right)$ , called the Robbin-Salamon index, see \cite{RobbinSalamon1,RobbinSalamon2}. The \textit{Conley-Zehnder index} of the capping $L$-orbit $\xbar$ of $H$ is defined by\footnote{Note that this definition is independent on the trivialization $\Phi$.}:
\[\mu_{CZ}(\xbar)=\frac{n}{2}+\ind\left(\Lambda_1^{\Phi},\Lambda_2^{\Phi}\right)\in \ZZZ\, .\]

We say that $H$ is \textit{admissible} with respect to the action window $(a,b]$ when all the capping $L$-orbits with action in $(a,b]$ are non-degenerate. In this case, the $\FFF_2$-vector space $CF_*^{(a,b]}(L,\omega;H)$ is graded by the Conley-Zehnder index. This condition is generic: admissible Hamiltonians for $L$ form a dense subset of $C^2_c(M\times [0,1],\RRR)$.  
\begin{prop}\label{prop:GenericHamiltonian}
Let $L$ be a compact Lagrangian submanifold of $(M,\omega)$, and let $H$ be a Hamiltonian. Then, there exists a Hamiltonian isotopy $\{g_t\}$ supported in a sufficiently small neighborhood of $L$, such that $\{\varphi_t^H\circ g_t\}$ has only non-degenerate $L$-orbits. Moreover, if $\{g_t\}$ is generated by $F$, then $\{\varphi_t^H\circ g_t\}$ is generated by $H_t(x)+F_t({(\varphi_t^H)}^{-1}(x))$, where $F$ can be chosen sufficiently $C^2$-small.
\end{prop}
\subsubsection{The boundary operator.}\label{subsubsection:BoundaryOperator}
Given two capping $L$-orbits $\xbar$ and $\ybar$ of $H$, let $\Mcc(\xbar,\ybar ;H,J)$ denote the space of Floer connecting strips $u$ from $x$ to $y$ with $\xbar\sharp u=\ybar$ for the constant homotopy $H_s=H$, $J_s=J$. The space $\Mcc(\xbar,\ybar;H,J)$ is endowed with the topology of $C^{2}$-convergence on compact subsets of $\BBB$. Explicitly: 
\[\Mcc(\xbar,\ybar;H,J) =\left\{ u:\BBB\rightarrow M,\, \mbox{ such that }\begin{matrix} \partial_su+J_t(\partial_tu-X_t)=0\\ \forall s\in \RRR,\, u(s,0),u(s,1)\in L,\\ \lim_{s\rightarrow -\infty}u(s,t)=x(t),\\ \lim_{s\rightarrow +\infty}u(s,t)=y(t),\\ \mbox{and } \xbar\sharp u =\ybar \end{matrix} \right\}\, .\]

Here, the $\SSS^1$-family $J=\{J_t\}$ is assumed to meet the transversality conditions required for $\Mcc(\xbar,\ybar,H,J)$ with $a<\Acc_H(\ybar)<\Acc_H (\xbar)\leq b$ to be manifolds of dimension 
\[\dim \Mcc(\xbar,\ybar;H,J)=\mu_{ CZ }(\xbar)-\mu_{CZ}(\ybar)\, .\]
Assume $b-a<\hbar_L(J_0)$, and choose the $ \SSS^1$-parametrized family $\{J_t\}$ inside a fixed simply connected neighborhood $\Ucc(J_0) \subset \{ J\in \Icc_A^2(J_0),\, \hbar_L(J)>b-a\}$ of $J_0$. Remember there is an $\RRR$-action operating by translation on the $s$-variable. Set
\[\Mchat(\xbar,\ybar;H,J)=\Mcc(\xbar,\ybar;H,J)/\RRR\, .\]
For each element $u\in \Mcc(\xbar,\ybar;H,J)$, lemma \ref{lemma:A3} gives
\[E(u)=\Acc_H(\xbar)-\Acc_H(\ybar)\leq b-a<\hbar_L(J_t),\, \forall t\, .\]
Thus, no bubbling off of holomorphic disks can occur in the limit set of a sequence~$(u_n)$ of $\Mcc(\xbar,\ybar;H,J)$. Bubbling off of holomorphic spheres can be generically avoided for one- and two-dimensional components of $\Mcc(\xbar,\ybar;H,J)$. Standard compactness arguments, already presented in Floer~\cite{Floer88,Floer88bis,Floer88ter}, imply:
\begin{itemize}
\item Whenever $\mu_{CZ}(\xbar)-\mu_{CZ}(\ybar)=1$, the zero dimensional manifold $\Mchat(\xbar,\ybar;H,J)$ is compact then finite;
\item Whenever $\mu_{CZ}(\xbar)-\mu_{CZ}(\zbar)=2$, the one-dimensional manifold $\Mchat(\xbar,\zbar;H,J)$ can be compatified as a cobordism between the empty set and the union of
\begin{align}  & \Mchat(\xbar,\ybar;H,J)\times \Mchat(\ybar,\zbar;H,J) & \mbox{ for } & \mu_{CZ}(\ybar)=\mu_{CZ}(\zbar)+1\, .\label{sets:McxyXMcyz}\end{align}
\end{itemize} 
For instance, see (\cite{Floer88}, section 2).\espace

Considering those observations, the following operator $\partial$ is well defined: 
\[\partial:\begin{matrix} CF_k^{(a,b]}(L,\omega;H) & \longrightarrow & CF_{k-1}^{(a,b]}(L,\omega;H)\\ \xbar & \longmapsto & \sum \sharp_2 \Mchat(\xbar,\ybar;H,J)\, y\, , \end{matrix}\]
where $\sharp_2$ denotes the number of elements mod 2. The coefficient behind $\zbar$ in the expression of $\partial^2 \xbar$ is exactly the cardinal of the set (\ref{sets:McxyXMcyz}), which is even. Thus, $\partial^2=0$. The homology of the chain complex $(CF_*^{(a,b]}(L,\omega;H),\partial)$ is the Floer homology groups:
\begin{align}HF^{(a,b]}_*(L,\omega;H,J_0)\, .\label{eqn:DefinitionHFab}\end{align}

Note that, for $a<b<c$ with $c-a<\hbar$, the short exact sequence
\[0\rightarrow CF_*^{(a,b]}\rightarrow CF_*^{(a,c]}\rightarrow CF_*^{(b,c]}\rightarrow 0\]
induces in homology a long exact sequence
\[HF_*^{(a,b]}\rightarrow HF_*^{(a,c]}\rightarrow HF_*^{(b,c]}\rightarrow HF_{*-1}^{(a,b]}\, .\] 

Note that  the dependence in the small peturbation $J$ is drawn up in the notations~(\ref{eqn:DefinitionHFab}). Indeed, the resulting homology groups are independent on this perturbation up to a unique isomorphism (see below). It seems important to choose $J$ in a contractible neighborhood $\Ucc(J_0)$, does it ?
\subsection{Definition of the continuation maps.}\label{Subsection:ContinuationMaps} Given two admissible pairs $(H_-,J_-)$ and $(H_+,J_+)$, the continuation map is a morphism defined by a compact homotopy $\{H_s\}$ from $H_-$ to $H_+$. We denote by $\Mcc(\xbar_-,\xbar_+;\{H_s\},\{J_s\})$ the space of the Floer continuation strips $u$ (for the compact homotopy $(H_s,J_s)$), from $x_-$ to $x_+$, with $\xbar_-\sharp u=\xbar_+$. Here, the homotopy $\{J_s\}$ goes from $J_-$ to $J_+$ and lies in $\Ucc(J_0)$. It is chosen to meet all the required transversality conditions for the spaces $\Mcc(\xbar_-,\xbar_+,\{H_s\},\{J_s\})$ to be manifolds of dimension $\mu_{CZ}(\xbar_-)-\mu_{CZ}(\xbar_+)$.
\begin{prop}\label{prop:DefinitionContinuationMap}
Say that $\{H_s\}$ is a \textit{$C$-homotopy}\footnote{Word introduced by Ginzburg~\cite{Ginzburg}.} when $\alpha_+(\partial_sH_s)<C$. The map
\[\Psi:\begin{matrix}CF_k^{(a,b]}(L,\omega;H_-) & \longrightarrow & CF_k^{(a+C,b+C]}(L,\omega;H_+)\\ \xbar_- & \longmapsto & \sum \sharp_2\Mcc(\xbar_-,\ybar_+;\{H_s\},\{J_s\})\, \ybar\, . \end{matrix}\]
is well-defined and commutes with the boundary operators.\end{prop}
\begin{proof} Let $\xbar_-$ and $\ybar_+$ be two capping $L$-orbits respectively of $H_-$ and $H_+$ such that $a<\Acc_{H^-}(\xbar_-)\leq b$ and $a+C<\Acc_{H^+}(\ybar_+)\leq b+C$. Lemma \ref{lemma:A3} shows that the energies of elements in $\Mcc(\xbar_-,\ybar_+,\{H_s\},\{J_s\})$ are uniformly bounded by $b-a$. Indeed, 
\[E(u)\leq \Acc_{H^-}(\xbar_-)-\Acc_{H^+}(\ybar_+)+\alpha_+(\partial_s H_s)\leq b-a-C+C=b-a\, .\]
Recall that $b-a<\hbar_L(J_z)$ for all $z\in \BBB$. So, as before, no bubbling off of holomorphic spheres or disks can appear in the limit set of a sequence $(u_n)$ in $\Mcc(\xbar_-,\ybar_+;\{H_s\},\{J_s\})$. Figure \ref{Figure:Explicative} helps to understand the following arguments. Up to an extraction of a subsequence, $(u_n)$ converges to a Floer continuation strip $v$ for the compact homotopy $(H_s,J_s)$, with finite energy. It goes from~$\ybar_-$ to~$\xbar_+$ with $\Acc_{H^-}(\ybar_-)\leq \Acc_{H^-}(\xbar_-)\leq b$ and $\Acc_{H^+}(\xbar_+)\geq \Acc_{H^+}(\ybar)>a+C$. As $\Acc_{H_+}(\xbar_+)-\Acc_{H_-}(\ybar_-)\leq C$, we immediately get $\Acc_{H^+}(\xbar_+)\leq b+C$ and $\Acc_{H^-}(\ybar_-)>a$. The actions of $\ybar_-$ and $\xbar_+$ belong respectively to the action windows $(a,b]$ and $(a+C,b+C]$. Thus, they are non-degenerate. As in \cite{Salamon,Floer88}, the limits of the sequences $(s_n\cdot u_n)$ with $s_n\rightarrow -\infty$ (resp. $+\infty$) form a "broken" Floer continuation strip from $\xbar_-$ to $\ybar_-$ (resp. from $\xbar_+$ to $\ybar_+$). Standard considerations on the index give:
\begin{itemize}
\item Whenever $\mu_{CZ}(\xbar_-)=\mu_{CZ}(\ybar_+)$, the zero-dimensional manifold $\Mcc(\xbar_-,\ybar_+;\{H_s\},\{J_s\})$ is compact then finite ;
\item Whenever $\mu_{CZ}(\xbar_-)=\mu_{CZ}(\ybar_+)+1$, non-compact components of the one-dimensional manifold $\Mcc(\xbar_-,\ybar_+;\{H_s\},\{J_s\})$ can be compactified in a cobordism between the sets:
\begin{align} & \Mchat(\xbar_-,\ybar_-;H_-,J_-)\times \Mcc(\ybar_-,\ybar_+;\{H_s\},\{J_s\}) & \mbox{ for } & \mu_{CZ}(\ybar_-)=\mu_{CZ}(\ybar_+)\\
\mbox{and } & \Mcc(\xbar_-,\xbar_+;\{H_s\},\{J_s\})\times \Mchat(\xbar_+,\ybar_+;H_+,J_+) & \mbox{ for } & \mu_{CZ}(\xbar_-)=\mu_{CZ}(\xbar_+)\, .\end{align}
\end{itemize} 
The first point shows that the definition of $\psi$ makes sense. The second point can be algebraically translated into $\partial\Psi=\Psi\partial$.\end{proof}
The map $\Psi$ induces a morphism in homology, called \textit{the continuation morphism}:
\begin{align} \Psi:HF_*^{(a,b]}(L,\omega;H_-,J_0)\rightarrow HF_*^{(a+C,b+C]}(L,\omega;H_+,J_0)\, .\label{Continuation}\end{align}
We point out that \textit{the continuation morphism $\Psi$ does not depend on the $C$-homotopy $(H_s,J_s)$ used to define it}. Moreover, \textit{the composition of two continuation morphisms is equal to the continuation morphism,} with the good shift in the action window.
\begin{figure}
\centering
\begin{picture}(14.5,11.5)
\put(-0.5,0){\vector(0,1){11.5}} \put(-3,11){Energy} 
\put(-0.7,9.5){\line(1,0){0.4}} \put(-0.7,4.5){\line(1,0){0.4}} \put(-0.7,5.5){\line(1,0){0.4}} \put(-0.7,0.5){\line(1,0){0.4}} 
\put(2,9.5){\line(0,-1){5}} \put(1.8,9.5){\line(1,0){2.4}} \put(4,9.5){\line(0,-1){5}} \put(1.8,4.5){\line(1,0){2.4}} \put(12,5.5){\line(0,-1){5}} \put(11.8,5.5){\line(1,0){2.4}} \put(14,5.5){\line(0,-1){5}}
\put(11.8,0.5){\line(1,0){2.4}} 
\put(2.8,9.7){$\xbar_-$} \put(2.8,4.7){$\ybar_-$} \put(12.8,5.7){$\xbar_+$} \put(12.8,0.7){$\ybar_+$}
\put(15,1.5){\vector(0,1){7}} \put(15.5,5){$\leq C$}
\qbezier(2,4.5)(2,1.5)(5,1.5)  \qbezier(5,3.5)(4,3.5)(4,4.5)
\qbezier(5,1.5)(7,1.5)(8.5,4) \qbezier(7.5,6)(6,3.5)(5,3.5)
\qbezier(8.5,4)(10,6.5)(11,6.5) \qbezier(11,8.5)(9,8.5)(7.5,6)
\qbezier(11,6.5)(12,6.5)(12,5.5) \qbezier(14,5.5)(14,8.5)(11,8.5)  
\put(-3.5,9.5){$\Acc_{H^-}(\xbar_-)$} \put(-3.5,5.5){$\Acc_{H^+}(\xbar_+)$} \put(-3.5,4){$\Acc_{H^-}(\ybar_-)$} \put(-3.5,0.5){$\Acc_{H^+}(\ybar_+)$} 
\end{picture}
\caption{Proof of proposition \ref{prop:DefinitionContinuationMap}}\label{Figure:Explicative}
\end{figure}
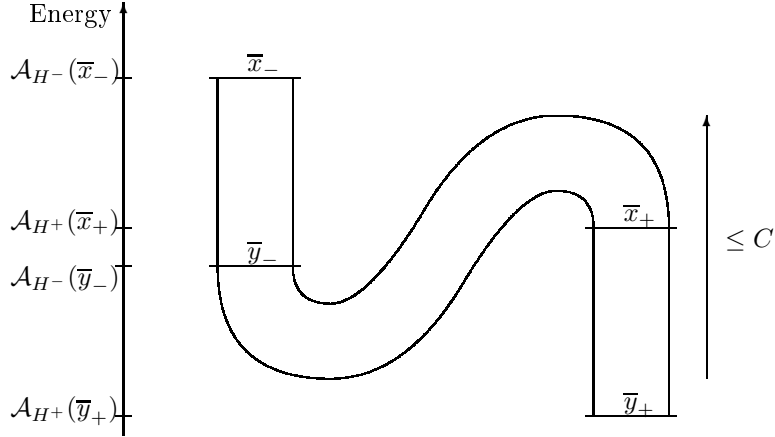

If $(H_s,J_s)$ and $(H_s',J_s')$ are two $C$-homotopies satisfying the required transversality conditions, then $H^r_{s,t}=(1-r)H_s+rH_s'$ is an homotopy of $C$-homotopies. Let $\Mcc(\xbar_-,\ybar_+;\{H_s^r\},\{J_s^r\})$ be the space of pairs $(r,u)$ where $r\in [0,1]$ and $u$ is a Floer continuation strip for the compact homotopy $(H_{s,t}^r,J_{s,t}^r)$ from $\xbar_-$ to $\ybar_+$. Here, the parametrized family $\{J^r_{s,t}\}$ of compact homotopies is chosen inside $\Ucc(J_0)$ so that the spaces $\Mcc(\xbar^-,\ybar^+;\{H_s^r\},\{J_s^r\})$ are smooth manifolds of dimension $\mu_{CZ}(\xbar_-)-\mu_{CZ}(\ybar_+)+1$. Once again, no bubbling off of holomorphic disks or spheres can occur. By counting $\sharp_2\Mcc(\xbar_-,\ybar_+;\{H_{s,t}^r\},\{J_{s,t}^r\})$ for $\mu_{CZ}(\ybar_+)=\mu_{CZ}(\xbar_-)+1$, we easily construct an homotopy map between the chain maps defined by $H_s$ and $H_s'$. The argument is classic, but as above, the reader would have to check that the involved capping orbits belong to the expected intervals.

We do not give the detailed proofs. Note that the continuation morphisms defined by constant homotopies $H_s=H$ are isomorphisms. We deduce that the Floer homology groups (\ref{eqn:DefinitionHFab}) are independent on the choice of the small perturbation of $J_0$, as announced in the end of the subsection \ref{Subsection:FilteredLagrangianFloerHomology}.\espace

Without proof, we assert here a result mentioned in (\cite{Ginzburg}, section 3.2.3, result H3):\espace

\textit{Let $a^s$ and $b^s$ be compact homotopies from $a^-$ to $a^+$ and from $b^-$ to $b^+$, with $b^s-a^s<\hbar(J_0)$. Assume that $a^s$ and $b^s$ are not critical values of the action functional $\Acc_{H_s}$, where $\{H_s\}$ is a compact homotopy from $H_-$ to $H_+$. Then there exists an isomorphism 
\[HF_*^{(a^-,b^-)}(L,H_-,J_0)\rightarrow HF_*^{(a^+,b^+)}(L,H_+,J_0)\, .\]}
\subsection{Proof of Theorem \ref{thm:Chekhanov} for rational Lagrangian submanifolds.}\label{Subsection:ProofofTheoremRational}
A Lagrangian submanifold $L$ is called \textit{rational} when $\omega\pi_2(M,L)=2c\mathbf{Z}$ with $c>0$ (\cite{Polterovich}, definition 1.2). We prove Theorem \ref{thm:Chekhanov} where $\hbar_L$ is replaced by $c$. First, we may assume that $\|H\|_+=\|H\|_-<c$, by replacing $H_t$ by $H_t-f(t)$, with a suitable function $f$. Remark that the energy of a holomorphic sphere or disk is positive, thus greater or equal to $2c$. Thus, $2c\leq \hbar_L(J_0)\leq \hbar_L$.
\subsubsection{Morse theory.}
Let $f:L\rightarrow \RRR$ be a Morse function. The Morse-Smale complex $CM_*(f)$ of $f$ is the $\FFF_2$-vector space generated by the set of critical points of $f$, and graded by the Morse index. Fix a Riemannian metric $g$ on $L$. (We can always assume that $g$ is induced by the almost complex structure $J_0$.) Let $\nabla f$ be the gradient of $f$ with respect to $g$, and let $\{\psi_t^f\}$ be the anti-gradient flow of $f$. For each critical
point $x$ of $f$, the sets
\begin{align}
W^s(x,\nabla f) & =\left\{y\in L,\, \psi_t^f(y)\rightarrow x\right\}\, ,\\
W^u(x,\nabla f) & = \left\{y\in L,\, \psi_{-t}^f(y)\rightarrow x\right\}
\end{align}
are embedded disks, respectively called the stable and unstable manifolds at $x$ (\cite{Jost}, Corollary~6.3.1). Recall that the Morse index $\mu_M(x)$ is equal to the dimension of $W^u(x,\nabla f)$. Generically on~$g$ (and hence on~$J_0$), all the stable and unstable manifolds intersect pairwise transversally. In particular, whenever $\mu_M(x)=\mu_M(y)+1$, the manifold $W^u(x,\nabla f)\cap W^s(y,\nabla f)$ has dimension~$1$. If it is non-empty, then $f(y)<f(x)$. For $f(y)<a<f(x)$, it intersects transversally the one-codimensional manifold~$f^{-1}(a)$, and the intersection $\Mchat(x,y;f,g)$ is a finite set, well-defined up to a unique bijection obtained by following the anti-gradient flow of $f$ (\cite{Jost}, section 6.5). The Morse boundary operator $\partial_g$ is defined as follows: 
\[\partial_g:\begin{matrix} CM_k(f) & \longrightarrow & CM_{k-1}(f)\\ x & \longmapsto & \sum \sharp_2 \Mchat(x,y;f,g)\, y\, . \end{matrix}\]

We have: $\partial_g^2=0$. The homology of the complex $\left(CM_*(f),\partial_g\right)$ is denoted by: \[HM_*(f,g)\, ,\] which is independent to $(f,g)$ up to a unique isomorphism, obtained by continuation as in Floer theory (\cite{Jost}, section 6.7). Those homology groups are isomorphic to the singular homology groups of $L$, which leads to the Morse inequalities (\cite{Jost}, section 6.10).\espace

Floer theory, presented in subsection \ref{Subsection:FilteredLagrangianFloerHomology}, may be viewed as an adaptation of the Morse theory for the action functional $\Acc_H$. Beyound the well-known analogy, there exists a deep link between Floer homology and Morse homology. From \cite{Weinstein71,Weinstein}, recall:
\begin{thm}[Weinstein (\cite{Weinstein71}, Theorem 6.1)]
For a sufficiently small $r>0$, there exists a symplectomorphism from $T^*_rL$ onto an open neighborhood $U$ of $L$, sending the zero section onto~$L$ as the identity.
\end{thm}
By abuse of notations, we denote a point in $U$ by its coordinates $(p,q)$ in $T^*_rL$. Fix a non-increasing function $\sigma:[0,r]\rightarrow [0,1]$ equal to $1$ on $[0,r/3]$, and to $0$ on $[2r/3,r]$. Set
\[K(z)=\left\{\begin{matrix} \epsilon f(q)\sigma(|p|) & \mbox{ if } z=(q,p)\in U\, ,\\ 0 & \mbox{ otherwise.} \end{matrix} \right.\]
For $\epsilon<r/3\|df\|$, the Hamiltonian flow of $K$ maps $L$ to the graph of $\epsilon df$ (a Lagrangian submanifold of $T^*_rL$ viewed in $U$). Thus, the $L$-orbits of $K$ are constant and equal to the critical points of~$f$. Each critical point $x$ of~$f$ can be completed with a disk $w$ bounded by $L$ to form a capping $L$-orbit $[x,w]$. Assuming $\epsilon \|f\|< b<c$, its action $\epsilon f(x)-\omega(w)$ belongs to the interval $(-c,b]$ iff $\omega(w)=0$. The Conley-Zehnder and Morse indices are equal. It thus follows that:
\[CF_*^{(-c,b]}(L,\omega;K,J_0)=CM_*(f,g)\left(\otimes \Lambda^{loc}_L\right)\, .\]
where $\Lambda^{loc}_L$ is a "local version" of Novikov ring.

Let $u:\BBB\rightarrow M$ be a Floer continuation strip for the constant homotopy $(K,J)$ from $\xbar$ to $\ybar$, with $-c<\Acc_K(\ybar)\leq\Acc_K(\xbar)\leq b$. Then, the energy of $u$ is equal to $\epsilon(f(x)-f(y))\leq \epsilon\|f\|$. For $\epsilon>0$ small enough, $u$ must lie inside $U_r$ (see Proposition \ref{prop:BoundsEnergy}). Thus, the maximum principle implies that $u$ must lie on $L$, and hence is constant in $t$.

In other words, $u(s,t)=v(s)$, where $v$ is a anti-gradient flow line of $f$. Thus,
\[HF_*^{(-c,b]}(L,\omega;K,J_0)=HM_*(f,g)\left(\otimes \Lambda^{loc}_L\right)\, .\]
\subsubsection{The factorization.}
The key to proving theorem \ref{thm:Chekhanov} is a factorization of the identity on $HM_*(f,g)\otimes \Lambda^{loc}_L$ through $HF_*^{(-c,b]}(L,H,\omega,J_0)$. This factorization will be obtained by adapting the continuation maps defined in subsection \ref{Subsection:ContinuationMaps}.\espace

Fix $\epsilon<c-C_-$ and $\epsilon<c-C_+$, and then $b>\epsilon+C_+$ and $b>\epsilon+C_-$, with $C_{\pm}=\|H-K\|_{\pm}<a$. Hence, we have $b+C_-<2c-\epsilon$. Let $K_s$ be the linear homotopy from $K$ to $H$ defined by \[K_{s,t}=\beta(s)H_t+(1-\beta(s))K_t\, .\] Fix $\{J_s\}$ be a compact homotopy from $J_0$ to $\{J_t\}$ and $\{J_s'\}$ be a compact homotopy from $\{J_t\}$ to $J_0$ such that the pairs $(K_s,J_s)$ and $(K_{-s},J_s')$ meet the required transversality conditions. Then, we consider the continuation maps defined by $(K_s,J_s)$ and $(K_{-s},J_{s}')$ 
\begin{align} \Psi_1:& CF_*^{(-c,b]}(L,K,\omega,J_0) \longrightarrow CF_*^{(-c,b]}(L,H,\omega,J_0)\\ \Psi_2: &CF_*^{(-c,b]}(L,H,\omega,J_0) \longrightarrow CF_*^{(-c,b]}(L,K,\omega,J_0)\, . \end{align}
Note that the present situation differs from the general case, as we do not translate the action windows. We assert:
\begin{prop}\label{prop:2.4} The maps $\Psi_1$ and $\Psi_2$ commute with the boundary operators. Moreover, the induced map in homology go inside the following commutative diagram: \[\xymatrix{ & HF_*^{(-c,b]}(L,\omega;H,J_0)\ar[rd]^{\Psi_2} & \\ HF_*^{(-c,b]}(L,\omega;K,J_0)\ar[ru]^{\Psi_1}\ar[rr]_{\id} & & HF_*^{(-c,b]}(L,\omega;K,J_0)\, . } \] \end{prop}
The proof is given step by step.

\begin{proof}[Step $1$. The map $\Psi_1$ commutes with the boundary operators.] The arguments are similar to those previously presented. We just have to check that the involved capping $L$-orbits in figure \ref{Figure:Explicative} belong to the action window $(-a,b]$. The two problematic configurations are the following:
\begin{enumerate}
\item Let $\xbar$, $\ybar$ and $\zbar$ be capping $L$-orbits respectively of $K$, $H$ and $H$. Fix a Floer continuation strip for the homotopy $(K_s,J_s)$ from $\xbar$ to $\ybar$ and a Floer continuation strip for the constant homotopy $(H,J)$ from $\ybar$ to $\zbar$. Assume $\Acc_K(\xbar)$ and $\Acc_H(\zbar)$ to belong to $(-c,b]$. Then, the estimates (\ref{Eqn:EstimateEnergies}) give: $-c\leq \Acc_H(\zbar)\leq \Acc_H(\ybar)$ and $\Acc_H(\ybar)\leq \Acc_{K}(\xbar)+C_+$. Recall $\Acc_K(\xbar)\leq \epsilon$ and $b>C_++\epsilon$. Thus, $\Acc_H(\ybar)$ belongs to the action window $(-c,b]$.
\item Now, let $\xbar$, $\ybar$ and $\zbar$ be capping $L$-orbits respectively of $K$, $K$ and $H$. Fix a Floer continuation strip for the constant homotopy $(K,J)$ from $\xbar$ to $\ybar$ and a Floer continuation strip for the  homotopy $(H_s,J_s)$ from $\ybar$ to $\zbar$. Once again, assume $\Acc_K(\xbar)$ and $\Acc_H(\zbar)$ to belong to $(-c,b]$. Then the estimates (\ref{Eqn:EstimateEnergies}) give: $\Acc_K(\ybar)\leq \Acc_K(\xbar)$, and $\Acc_K(\ybar)\geq \Acc_H(\zbar)-C_+\geq -c-C_+>\epsilon-2c$. Recall that there is no capping $L$-orbit of $K$ whose action is between $\epsilon-2c$ and $-\epsilon$. Thus, $\Acc_K(\ybar)$ must belong to the action window $(-c,b]$.
\end{enumerate}
Then, the standard arguments show that the map $\Psi_1$ commutes with the boundary operators.\end{proof}
\begin{proof}[Step 2. The map $\Psi_2$ commutes with the boundary operators,] for similar reasons.\end{proof}
\begin{proof}[Step 3. There exists an homotopy map between the composition $\Psi_2\circ \Psi_1$ and the identity.] Introduce the parametrized family $\{K_{s,t}^R\}$ of compact homotopies from $K$ to $K$:
\[K_{s,t}^R=\beta(s+R)\beta(-s+R)H_t+(1-\beta(s+R)\beta(-s+R))K\, .\]
Note $X_{s,t}^R$ its Hamiltonian vector field. Set $\KKK=\{K_{s,t}^R\}$, and choice a generic data $\JJJ=\{J_{s,t}^R\}$ with good asymptotic behavior. For any pair $(x,z)$ of critical points of $f$, the space
\[\Mcc(x,z;\KKK,\JJJ)=\left\{(R,u),\, R\in \RRR\mbox{ and } u\in \Mcc(x,z;\{K_{s,t}^R\},\{J_{s,t}^R\})\right\}\]
is generically a manifold of dimension $\mu_M(x)-\mu_M(z)+1$. Generically,
\begin{itemize}
\item Whenever $\mu_M(z)=\mu_M(x)+1$, the zero-dimensional manifold $\Mcc(x,z;\KKK,\JJJ)$ is compact then finite ;
\item Whenever $\mu_M(z)=\mu_M(x)$, the non-compact components of the one-dimensional manifold $\Mcc(x,z;\KKK,\JJJ)$ can be completed into a cobordism between the sets
\begin{align}
 & \Mcc(x,z;K,J_0)=\left\{\begin{matrix}\{x\} & \mbox{if } x=z\\ \emptyset & \mbox{otherwise,}\end{matrix}\right. & & \label{set:01}\\
 & \Mcc(x,\ybar;\{K_{s,t}\},\{J_{s,t}\})\times \Mcc(\ybar,z; \{K_{-s,t}\},\{J'_{s,t}\}) &  & \mbox{for } \mu_{CZ}(\ybar)=\mu_M(x)\, , \label{set:02}\\
 & \Mcc(x,y;\KKK,\JJJ)\times \Mchat(y,z;K,J_0) & & \mbox{for } \mu_{M}(y)=\mu_M(z)+1\, ,\label{set:03}\\
 & \Mchat(x,y;K,J_0)\times \Mcc(y,z;\KKK,\JJJ) & & \mbox{for } \mu_M(y)=\mu_M(z)-1\, .\label{set:04}
\end{align}
\end{itemize} The capping orbits of $H$ appearing in (\ref{set:02}) have actions in $(-c,b]$.

Then, set provisionally 
\[\Gamma:\begin{matrix} CM_k^(f) & \rightarrow & CM_{k+1}(f)\\ x & \mapsto & \sum \sharp_2 \Mcc(x,z;\KKK,\JJJ_0)\, z\, .\end{matrix}\]
This map is well-defined, due to the first point. The existence of the cobordism described above can be algebraically translated into
\[\Psi_2\circ\Psi_1-\id=d\Gamma+\Gamma d\, .\]
Indeed, counting the elements in the sets (\ref{set:02}) gives $\Psi_2\circ \Psi_1$. The coefficient behind $z$ in the expression of $d\Gamma x$ (resp. $\Gamma d x$) is exactly the number modulo 2 of elements in sets (\ref{set:04}) (resp. in sets (\ref{set:03})). Thus, the map $\Gamma$ is an homotopy map between the identity and $\Psi_2\circ \Psi_1$. We have done.
\end{proof}
See appendix~\ref{Subsection:RemarksOriginalProofChekhanov} for remarks on the slight modifications to the original proof of Chekhanov.

\newpage
\section{Kerman's proof.}\label{Section:Kerman'sProof}
In this section, Theorem \ref{thm:Chekhanov} is proved for all compact Lagrangian submanifolds. In \cite{Kerman1}, Kerman explains how the identity $CM(f)\rightarrow CM(f)$ factors through $CF(L,\omega;H)$. This factorisation can be expressed in homological terms as follows.
\subsection{Definition of the relative PSS maps.}\label{Subsection:DefinitionRelativePSS}
\subsubsection{\!\!\!\!} Throughout this section, we fix a non-decreasing map $\beta:\RRR\rightarrow [0,1]$ equal to $0$ when $s\ll 0$ and to $1$ when $s\gg 0$. The precise definition of $\beta$ has no importance\footnote{Nevertheless, the monotonicity of $\beta$ is a crucial point to get the estimates mentioned on the energies.}. Let $X$ and $Y$ be two smooth submanifolds of $L$. Set:
\begin{align}
\Ncc_r(\ybar;H,J) & = \left\{d:\RRR\times \SSS^1\rightarrow M,\, \begin{matrix} \partial_sd+J_{s,t}(d)\left[\partial_td-\beta(s)X_t(d)\right]=0\\ \lim_{s\rightarrow \infty} d(s,t)=x(t) \mbox{ and } [x,d]=\xbar \end{matrix}\right\}\, ,\label{set:NrxHJ}\\
\Ncc_l(\xbar;H,J) & =\left\{e:\RRR\times \SSS^1\rightarrow M,\, \begin{matrix} \partial_se+J_{-s,t}(u)\left[\partial_te-\beta(-s)X_t(e)\right]=0 \\ \lim_{s\rightarrow -\infty} e(s,t)=x(t) \mbox{ and } [x,-e]=\xbar \end{matrix}\right\}\, ,\label{set:NlxHJ}\\ 
\Ncc(X,\ybar;H,J) & = \left\{(x,d)\in X\times \Ncc_r(\ybar;H,J),\, \lim_{s\rightarrow -\infty}d(s,t)=x\right\}\, ,\label{set:NrXxHJ}\\ 
\mbox{and }\Ncc(\xbar,Y;H,J) & = \left\{(e,y)\in \Ncc_l(\xbar;H,J)\times Y,\, \lim_{s\rightarrow +\infty} e(s,t)=y\right\}\, .\label{set:NlXxHJ} \end{align}

Elements of $\Ncc_r(\xbar;H,J)$ may be viewed as holomorphic half-disks with an Hamiltonian perturbation on their boundaries. In the definition of $\Ncc_l(\xbar;H,J)$, the notation $-e$ denotes the map $(s,t)\mapsto e(-s,t)$, which may be viewed as an anti-holomorphic half-disk with an Hamiltonian perturbation on its boundary. Lemma \ref{lemma:A3} gives the following estimates on the energies of perturbed (anti-)holomorphic half-disks:
\begin{align}
\forall d\in \Ncc_r(\xbar; H,J),\quad 0\leq E(d)\leq & -\Acc_H(x,d)+\|H\|_+\\
\forall e\in \Ncc_l(\xbar;H,J),\quad 0\leq E(e)\leq &
\Acc_H(x,-e)+\|H\|_-\, .
\end{align}
\noindent The expected dimensions are:
\begin{align*}
\dim \Ncc_r(\ybar;H,J) & = n-\mu_{CZ}(\ybar)\, ,\\
\dim \Ncc_l(\xbar;H,J) & = \mu_{CZ}(\xbar)\, ,\\
\dim \Ncc(X,\ybar;H,J) & = \dim X-\mu_{CZ}(\xbar)\, ,\\
\dim \Ncc(\xbar,Y;H,J) & = \mu_{CZ}(\ybar)+\dim Y-n\, .
\end{align*}
For generic choices, those spaces are well-defined manifolds.

\subsubsection{\!\!\!\!} Recall $\|H\|_+<b$, $\|H\|_-<c$ and $b+c<\hbar$. Set:
\begin{align}
\Phi_1&:\begin{matrix} CM_{k}(f,g) & \longrightarrow & CF_k^{(-c,b]}(L,\omega;H)\\ y & \longmapsto & \sum \sharp \Ncc\left(W^u(y,\nabla f), \xbar;H,J\right) \, \xbar\, , \end{matrix}\label{Eqn:Sets27}\\
\mbox{and } \Phi_2&:\begin{matrix} CF_k^{(-c,b]}(L,\omega;H) & \longrightarrow & CM_{k}(f,g) \\ \xbar & \longmapsto & \sum \sharp_2 \Ncc\left(\xbar,W^s(y,\nabla f);H,J\right) \, y\, , \end{matrix}\label{Eqn:Sets28}\end{align}
\noindent Observe that the sets appearing in (\ref{Eqn:Sets27}) and (\ref{Eqn:Sets28}) are finite.
\begin{figure}
\centering
\includegraphics[height=55mm]{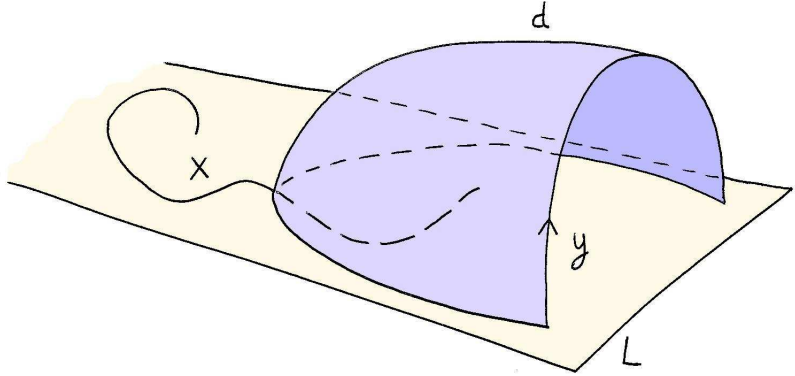}
\caption{Representations of elemens in the spaces (\ref{set:NrXxHJ}).}
\end{figure}
\begin{thm}
With the above notations, $\Phi_1$ and $\Phi_2$ are chain maps. The induced maps in homology are called the PSS maps: 
\begin{align*} \Phi_1: & HM_*(f,g)  \longrightarrow HF_*^{(-c,b]}(L,\omega;H,J_0) \\ \mbox{ and } \Phi_2 :& HF_*^{(-c,b]}(L,\omega;H,J_0) \longrightarrow HM_*(f,g)\, . \end{align*}
\end{thm}
\begin{proof} We only prove that $\Phi_2$ is a chain map. (The proof for $\Phi_1$ uses similar arguments.)
Fix a capping $L$-orbit $\xbar$ of $H$, with action in $(-c,b]$, and consider a critical point $y$ of $f$ with $\mu_M(y)=\mu_{CZ}(\xbar)-1$. Let $(u_n)$ be a sequence in $\Ncc_l(\xbar;H,J)$ with $\lim_{s\rightarrow \infty} u_n(s,t)=x_n\in W^u(y,\nabla f)$.

As $E(u_n)\leq \Acc_H(\xbar)+\|H\|_-<\hbar$, no bubbling off of holomorphic spheres or disks may occur in  the limit set of $(u_n)$. Up to an extraction, the sequence $(u_n)$ converges to a Floer continuation strip $v$ for the homotopy $\beta(-s)H$ from a capping $L$-orbit $\xbar '$ to a point $z$ of $L$, with $\Acc_H(\xbar ')\leq \Acc_H(\xbar)$. Moreover, we have $\Acc_H(\xbar ')\geq -\|H\|_-> -c$. As its action belongs to $(-c,b]$, the capping $L$-orbit $\xbar '$ is non-degenerate.

After an extraction if necessary, we obtain a broken Floer continuation strip from $\xbar$ to $\xbar '$ as the different limits of $(s_n\cdot u_n)$ for $s_n\rightarrow -\infty$. Note that
\[\lim_{T\rightarrow \infty} \lim_{n\rightarrow \infty}
\int_{T}^{\infty}\int_0^1|\partial_su_n|^2dsdt=0\, .\] It then follows that whenever $s_n\rightarrow \infty$, the sequence $(s_n\cdot u_n)$ converges to $z$ uniformly  on compact subsets of $\BBB$. Moreover, $z$ must be the limit of $(x_n)$.

Indeed, let $z'$ be an accumulation point of $(x_n)$. We may assume $x_n\rightarrow z'$ to simplify the notations. Choose $s_n>n$ such that $u(s_n,0)$ converges to $z'$. As $(s_n\cdot u_n)$ converges to $x$ uniformly on compact subsets, it follows that $z'=x$. In other words, the sequence $(x_n)$ has a unique accumulation point, namely $z$, and hence converges to $z$ as $L$ is compact. 

This limit~$z$ must belong to the adherence of $W^s(y,\nabla f)$. Thus, there exists a critical point $y'$ of $f$, with $\mu_M(y')\geq \mu_M(y)$, such that $x\in W^s(y',\nabla f)$. Moreover, there exists a broken antigradient flow from $y'$ to $y$. It then follows that $\Ncc(\xbar',W^u(y',\nabla f);H,J)$ is non empty. Thus, its dimension $\mu_{CZ}(\xbar')-\mu_M(y')$ must be nonnegative. Hence,
\[\mu_{CZ}(\xbar)\geq \mu_{CZ}(\xbar')\geq \mu_M(y')\geq
\mu_M(y)=\mu_{CZ}(\xbar)-1\, .\] Different cases must be considered: \begin{itemize}
\item $\mu_{CZ}(\xbar')\neq \mu_M(y')$ gives $\mu_{CZ}(\xbar')=\mu_{CZ}(\xbar)$ and $\mu_{M}(y')=\mu_M(y)$. The broken Floer continuation strip from $\xbar$ to $\xbar '$ we obtained above must be constant. Moreover, as the critical point $y'$ belongs to the adherence of $W^u(y,\nabla f)$, it must be equal to $y$. Then, the limit $v$ belongs to $\Ncc(\xbar,W^u(y,\nabla  f);H,J)$.
\item $\mu_{CZ}(\xbar)>\mu_{CZ}(\xbar')$ gives $\mu_{CZ}(\xbar')=\mu_M(y')=\mu_M(y)$. As above, $y=y'$. Moreover, the broken Floer continuation map from $\xbar$ to $\xbar'$ is of index $1$. Thus, there is no intermediate capping $L$-orbit of $H$, and we get an element of $\Mchat(\xbar,\xbar ';H,J)\times \Ncc(\xbar ',W^u(y,\nabla f);H,J)$. 
\item The last case to consider is the following: $\mu_{CZ}(\xbar)=\mu_{CZ}(\xbar')=\mu_M(y')$. Thus, $\xbar'=\xbar$, and the limit $v$ belong to the set $\Ncc_l(\xbar,W^u(y',\nabla
f);H,J)$.
\end{itemize}

Considering this study, the one-dimensional manifold $\Ncc_l(\xbar,W^u(y,\nabla f);H,J)$ may be compactified into a cobordism between:
\begin{align}
 & \Ncc(\xbar,W^u(y',\nabla f);H,J)\times \Mchat(y',y,\nabla f) & \mbox{ for }
 & \mu_M(y')=\mu_M(y)+1\, ,\\
\mbox{and } & \Mchat(\xbar,\xbar';H,J)\times \Ncc(\xbar',W^u(y,\nabla f);H,J) & \mbox{ for } & \mu_{CZ}(\xbar')=\mu_{CZ}(\xbar)-1\, .
\end{align}
Thus, we get:
\begin{align*}
\partial \Phi_2\xbar & =\sum_{\mu_{CZ}(\xbar)=\mu_M(y')=\mu_M(y)+1} \sharp_2 \Ncc(\xbar,W^u(y',\nabla f);H,J)\times \Mchat(y',y,\nabla f)\, y\\
 & = \sum_{\mu_M(y)=\mu_{CZ}(\xbar')=\mu_{CZ}(\xbar)-1}\sharp_2 \Mchat(\xbar,\xbar';H,J)\times \Ncc_l(\xbar',W^u(y,\nabla f);H,J)\, y\\
 & = \Phi_2\partial \xbar\, ,
\end{align*}
which proves that $\Phi_2$ commutes with the boundary operators, as wanted.
\end{proof}
\subsection{Proof of Theorem \ref{thm:Chekhanov}.}
In \cite{Kerman1}, Kerman proposed an approach to the Hamiltonian Floer theory under the quantum effects. We present here an adaptation of this approach for the Lagrangian Floer homology.
\begin{thm}
With the above notations, the composition of the PSS maps
\[HM_*(f,g)\rightarrow HF_*^{(-c,b]}(L,H,\omega)\rightarrow HM_*(f,g)\]
is the identity.
\end{thm}
\begin{proof}[Skretch of the proof.]
Let $y$ and $z$ be two critical points of $f$ with the same Morse index $k$. The one-dimensional manifold $\Ncc(W^s(y,\nabla f),W^u(z,\nabla f);\HHH,\JJJ)$ (introduced in section \ref{Section:Gromov'sProof}) can be compactified into a cobordism between:
\begin{align}
 & W^s(y,\nabla f)\cap W^u(y,\nabla f)=\left\{\begin{matrix}  \{y\} & \mbox{if } y=x\, ,\\  \emptyset & \mbox{otherwise.}  \end{matrix}\right. &  \label{Set:Tm3.2Id}\\
 & \Ncc(W^s(y,\nabla f),\xbar;H,J)\times \Ncc(\xbar,W^u(y,\nabla  f);H,J) & \mbox{ for } & \mu_{CZ}(\xbar)=k\, ,\label{Set:Thm3.2PhiPhi}\\
 &\Ncc(W^s(y,\nabla f),W^u(z',\nabla f);\HHH,\JJJ)\times \Mchat(z',z,\nabla f) & \mbox{ for } & \mu_M(z')=k+1\, ,\label{Set:Thm3.2dGamma}\\
\mbox{and } & \Mchat(y,y',\nabla f)\times \Ncc(W^s(y',\nabla f),W^u(z,\nabla f);\HHH,\JJJ) & \mbox{ for } & \mu_M(y')=k-1\, .\label{Set:Thm3.2Gammad}
\end{align}
Set provisionally\footnote{Up to the end of the proof.}:
\[\Gamma:\begin{matrix} CM_k(f) & \longrightarrow & CM_{k+1}(f)\\ x & \longmapsto & \sum \sharp_2\Ncc\left(W^u(x,\nabla f),W^s(z,\nabla f);\HHH,\JJJ\right)\, z\, . \end{matrix}\]

The above cobordism gives the following relation: \[\Phi_2\circ \Phi_1-\id = d\Gamma+\Gamma d\, .\] More precisely, $d\Gamma$ and $\Gamma d$ count respectively the number of elements in the finite sets (\ref{Set:Thm3.2dGamma}) and (\ref{Set:Thm3.2Gammad}). The map $\Phi_2\circ \Phi_1$ counts the elements in the sets (\ref{Set:Thm3.2PhiPhi}). We have done.
\end{proof}
\subsection{Equivalence of the previous proofs.}
The equivalence between Gromov's proof and Kerman's proof is clear. We explain how Kerman's proof is related to Chekanov's proof. For the notations, refer to
subsections \ref{Subsection:ProofofTheoremRational} and \ref{Subsection:DefinitionRelativePSS}.
\begin{thm}
Assume $L$ to be rational with $\omega\pi_2(M,L)=2c\ZZZ$. To simplify, assume that there exists no disk bounded by $L$ with zero Maslov index and non-zero symplectic area. Let $K$ and $b$ be as in subsection \ref{Subsection:ProofofTheoremRational}. The PSS maps $\Phi_1$ and $\Phi_2$ are equal to the continuation maps $\Psi_1$ and $\Psi_2$ obtained via a linear compact homotopy from $K$ to $H$ and from $H$ to $K$.
\end{thm}
\begin{proof} We only explain how to get an homotopy map between the chain maps $\Psi_1$ and $\Phi_1$. The arguments may be easily adapted to the pair $(\Psi_2,\Phi_2)$.

Let us consider the parametrized family $\{H_{s,t}^R\}$ of linear compact homotopies from $K$ to $H$ defined as follows: \[H_{s,t}^R=\beta(-s-R)K+\beta(s-R)H_t\, .\] Note $X_{s,t}^R$ its Hamiltonian vector field, and introduce the space:
\[\Mcc\left(x,\ybar; \{H_{s,t}^R\},\{J_{s,t}^R\}\right)=\left\{ (R,u),\, \mbox{ st } \begin{matrix} R\in \RRR_+,\\ \partial_su+J_{s,t}^R\left[\partial_tu-X_{s,t}^R\right]=0\\ \lim_{s\rightarrow -\infty}u(s,t)=x\\ \lim_{s\rightarrow +\infty}u(s,t)=y(t)\\ [y,u]=\ybar \end{matrix}\right\}\, .\]
The topology is given by the topology of the extended real half-line times the topology of $C^{2}$-convergence on compact sets. Here, the compact homotopy  $\{J_{s,t}^R\}$ is obtained by perturbing $\{J_{s+R,t}\}$. The perturbation is chosen for the space $\Ncc\left(x,\ybar;\HHH,\JJJ\right)$ to be a manifold of the expected dimension $\mu_{CZ}(\xbar)-\mu_{CZ}(\ybar)+1$.

For $(R,u)\in \Ncc\left(x,\ybar;\{H_{s,t}^R\},\{J_{s,t}^R\}\right)$, we have:
\[E(u)\leq f(x)-\Acc_H(\ybar)+\|f\|_-+\|H\|_+\, .\]
Consequently, no bubbling off of holomorphic disks may occur in the limit set. By classical arguments, for generic data, it follows that:
\begin{itemize}
\item When $\mu_{CZ}(\ybar)=\mu_M(x)+1$, the zero-dimensional manifold $\Mcc(x,\ybar;\HHH,\JJJ)$ is compact then finite.
\item When $\mu_{CZ}(\ybar)=\mu_{CZ}(x)$, the one-dimensional manifold $\Mcc(x,\ybar;\HHH,\JJJ)$ can be compactified into a cobordism between the sets:
\begin{align}
 & \Mcc(x,\ybar;\{H_s^0\},\{J_s\}) &  & ; \label{set3}\\
 & \Mcc(x,\zbar;\HHH,\JJJ)\times \Mcc(\zbar,\ybar;H,J) &  \mbox{for } & \zbar\in \Pcc^{(a,b]}(H,\omega)\, ;\label{set1}\\
 & \Mcc(x,z;f,g)\times\Mcc(z,\ybar;\HHH,\JJJ) & \mbox{for } &  z\in \crit(f), R>0\, ;\label{set2}\\
\mbox{and } & \Ncc\left(W^u(x,\nabla f), \xbar; H,J\right)\, .\label{set4} \end{align}
\end{itemize}

Let us justify the second assertion. Given a sequence $(R_n,u_n)$ in $\Mcc(x,\ybar,\{H_{s,t}^R\},\{ J_{s,t}^R\})$, we may assume that the sequence $(u_n)$ converges uniformly on compact subsets. Different possibilities are to be considered:\espace

\noindent \textit{Case $1$.} If $R_n\rightarrow 0$, the limit $v$ of $(u_n)$ is a Floer continuation map from $x$ to $\ybar$.\\

\noindent \textit{Case $2$.} Assume $R_n\rightarrow R$. Consider the limits of $s_n\cdot u_n$.
\begin{itemize} \item If $s_n\rightarrow -\infty$, we get a broken Floer continuation strip for the constant homotopy $K$ from $y$ to capping $L$-orbit $\xbar '$ of $K$ ; \item If $(s_n)$ is bounded, we get a Floer continuation strip $u$ from $\xbar '$ to a capping $L$-orbit $\ybar '$ of $H$, unique up to translation on the $s$-variable ; \item If $s_n\rightarrow \infty$, we get a broken Floer continuation strip for the constant homotopy $H$ from $\xbar'$ to $\xbar$, unique up to translation.
\end{itemize}
The estimates (\ref{Eqn:EstimateEnergies}) give successively $\Acc_K(\xbar')\leq \Acc_K(x)=f(x)\leq b$ ; $\Acc_H(\ybar ')\geq \Acc_H(\ybar)\geq -c$ and $\Acc_K(\xbar')-\Acc_H(\ybar ')+C\geq 0$ for $C=\|f\|_-+\|H\|_+$. Considering those inequalities, we obtain $\Acc_K(\xbar ')\geq \Acc_H(\ybar ')-C\geq -c-\|H\|_+-\epsilon>\epsilon-2c$. As there is no capping $L$-orbit with action in $(\epsilon-2c,-\epsilon)$, we get $\Acc_K(\xbar ')\geq -c$. In other hand, $\Acc_H(\ybar ')\leq \Acc_K(\xbar ')+C\leq \epsilon+C\leq b$.

For each involving capping $L$-orbits appearing in the limit set, its action belongs to $(-c,b]$. Considerations on the indices show that
\begin{itemize}
\item Either $\xbar=\xbar '$. In this case, we get a point in the sets (\ref{set1}) ;
 \item Either $\ybar=\ybar '$. In this case, we obtain a point in the sets (\ref{set2}).\espace
\end{itemize}

\noindent\textit{Case $3$}. Assume $R_n\rightarrow \infty$. The limit of $(u_n)$ gives a point in the set $\Ncc(W^u(x,\nabla f), \xbar;H,J)$. To conclude, standard gluing arguments are needed to prove that each point in the sets (\ref{set3}) (\ref{set1}), (\ref{set2}) and (\ref{set4}) may be obtained as a limit point.\espace

\noindent The first assertion shows that the following map is well-defined:
\[\Phi: \begin{matrix} CM_k(f,g) & \longrightarrow & CF_k(L,\omega;H)\\ x & \longmapsto & \sum \sharp_2\Mcc\left(x,\ybar; \HHH ,\JJJ\} \right)\, y\, . \end{matrix}\, .\] The second assertion can be algebraically translated by the equality: \[\Phi_1-\Psi_1 = \partial \Gamma+\Gamma\partial\, .\] We have done.
\end{proof}

\newpage
\section{Concluding remarks.}\label{Section:ConcludingRemarks}
\subsubsection{\!\!\!\!} Here are a few remarks on mistakes to avoid.\espace

{\bf 1.} Repeatedly in this paper, the proofs use one-dimensional cobordisms between two finite sets $A$ and $B$. Elements of $A$ and $B$ are geometric objects as strips, and they have energies. Estimates on the energies of the elements of $A$ do not imply any information on objects of $B$. Here are two reasons:

-- In general, the energy is non-constant along the cobordism;

-- The cobordism is simply given by a partition of $A\cup B$ into pairs, but each element of $B$ is not necessarly associated to an element of $A$. For instance, $A$ may be empty and $B$ is the boundary of a one-dimensional compact manifold.\espace

{\bf 2.} The localization of Lagrangian Floer homologies described in this paper is \textit{not} isomorphic to the Morse homology. The different maps described (continuation maps and PSS maps) are \textit{not} isomorphisms in general.\espace

{\bf 3.} The PSS maps could lead the reader to some confusions. Among capping orbits $\xbar$, some are homologically important, those for which $\Ncc_l(\xbar;H,J)$ and $\Ncc_r(\xbar;H,J)$ are non empty. Call them local. First, this definition \textit{explicitly} depends on the perturbation~$J$ of~$J_0$. A compact homotopy~$\{J_s\}$ defines a cobordism between $\Ncc_l(\xbar;H,J_-)$ and $\Ncc_l(\xbar;H,J_+)$, but $\Ncc_l(\xbar;H,J_-)\neq \emptyset$ does not imply $\Ncc_l(\xbar;H,J_+)\neq \emptyset$. For similar reasons, the property "local" depends on the capping half disk.

For instance, when $\mu_{CZ}(\xbar;H,J)=3$, the space $\Ncc_l(\xbar;H,J)$ is a three-dimensional manifold, and thus bounds a four-dimensional manifold. As $\Ncc_l(\xbar;H,J)$ is only defined up to cobordism, we get no information on $\xbar$.\espace

{\bf 4.} Nevertheless, for a fixed time-depending almost complex structure $J$, those "local" capping orbits generate some vector subspace $E$ of $CF_*(L,H)$. Counting the Floer continuation strips with index $0$ defines an operator $u$ on $E$. But, $u^2\neq 0$. The
reason is the following.

Consider a pair $(u,v)$ in the set $\Mchat(\xbar,\ybar;H,J)\times \Mchat(\ybar,\zbar;H,J)$, where $\xbar$ and $\zbar$ are "local". Is $\ybar$ "local" ? There is no way to know it. There is a cobordism between $\Ncc_r(\xbar;H,J)\times \Mchat(\xbar,\ybar;H,J)$ and $\Ncc_r(\ybar;H,J)$ ; but the second space can be empty. Once again, we cannot conclude. \espace
\subsubsection{Local Hamiltonian Floer homology.}
By similar arguments, a local version of the Hamiltonian Floer homology can be defined. At first sight, this local version seems less usefull as all the problems due to the presence of holomorphic spheres with negative indices can be avoided. The Hamiltonian Floer homology is well-defined for general compact symplectic manifolds.

For a compact symplectic manifold $(M,\omega)$, the Hamiltonian Floer homology can be viewed as the Lagrangian Floer homology of the diagonal $\Delta_M$. (Here, $\Delta_M$ is the Lagrangian submanifold of $\Mbar\times M$ collecting the pairs $(x,x)$.)

--- A disk $(u,v)$ bounded by $\Delta_M$ gives rise to a sphere $w:\CCC \PPP^1\rightarrow M$ obtained by gluing $v$ and $z\mapsto u(z/|z|^2)$. The Maslov index of $(u,v)$ is exactly the Chern number of $w$. If $(u,v)$ is a $(-J)\oplus J$-holomorphic disk, then $w$ is $J$-holomorphic sphere. Note that all the required transversality conditions can be satisfied by almost complex structures $(-J)\oplus J$. The minimal area of a holomorphic disk of $\Mbar\times M$ bounded by $\Delta_M$ is $\hbar$.

--- Take a Hamiltonian $H$ on $M$. A contractible $\Delta_M$-orbit $(y,x)$ of $\pi_2^*H$ corresponds to a contractible one-periodic orbit $x$ of $H$ with $x(0)=y$. It is readily seen to be non-degenerate  iff $1$ is not an eigenvaue of $d\varphi^H_1(x_0)$. A capping half-disk $(u,v)$ gives a disk $w$ bounded by a reparametrization of $x$ obtained as the gluing of the maps $v:\DDD_+\rightarrow M$ and $u':\DDD_-\rightarrow M$ ($u'(z)=u(\overline{z})$). As easily checked, a local maximum of a small Morse function has index $n$.

--- A similar discussion is needed to understand how to deal with the Floer continuation strips.
\begin{prop}
Let $(M,\omega)$ be a compact symplectic manifold. Let $b-a<\hbar$.
Then, there exists an almost complex structure $J_0$ such that
$b-a<\hbar_L(J_0)$. And
\[HF_*^{(a,b]}(M,H,\omega,J_0)=HF_*^{(a,b]}(\Delta,\pi_2^*H,\omega,J_0)\]
is well-defined. If $-\hbar/2<a<-\|H\|_-<\|H\|_+\leq b<\hbar/2$, then the
composition of the natural sequences
\[HM_*(M)\rightarrow HF_*^{(a,b]}(M,H,\omega,J_0)\rightarrow HM_*(M)\]
is the identity.
\end{prop}
The filtered version of the Lagrangian Floer homology is still available for certain non-compact Lagrangian submanifolds. For example, the diagonal of a geometrically bounded symplectic manifold is not compact in general, but a filtered Hamiltonian Floer homology can be defined for an action window $(0,a)$ and $a<\hbar$.

\newpage
\section{Appendices.}\label{Section:Appendices}
\subsection{Estimates on the energies.}\label{Subsection:EstimatesEnergies}
\subsubsection{\!\!\!\!}\label{Subsubsection:EstimatesEnergies0} Proving Theorem \ref{thm:Chekhanov} requires estimates on the energies of holomorphic curves and their Hamiltonian perturbations. Those estimates are now standard tools in symplectic topology. Most notations are introduced in section \ref{Section:Gromov'sProof}. In particular, conditions C1 and C2 are met for a $\omega$-compatible almost complex structure $J_0$.
\begin{lemma}\label{lemma:5.1}
Let $\Sigma$ be a Riemann surface (possibly with boundary). For a $\Sigma$-parametrized family $\JJJ=\{J_z\}$ of almost complex structures in $\Icc_A(J_0)$, a $\JJJ$-holomorphic curve $u:\Sigma\rightarrow M$ satisfies: \[E(u)\geq \frac{1}{A}\area(u)\, .\]
\end{lemma}
\begin{proof} Let $z=s+it$ be a local chart on the Riemann surface $\Sigma$. Then, $u^*\omega= \omega( \partial_su,\partial_tu)ds\wedge dt$, where:
\begin{align*}
\omega(\partial_su,\partial_tu)& =\omega(\partial_su,J\partial_su)^{1/2}\omega(\partial_tu,J\partial_tu)^{1/2}\\
 & \geq \frac{1}{A}\|\partial_su\|\cdot\|\partial_tu\| \\
 & \geq  \frac{1}{A} \sqrt{\|\partial_su\|^2\cdot  \|\partial_tu\|^2-<\partial_su|\partial_tu>^2}\, .
\end{align*}
\end{proof}
\begin{lemma}[Viterbo (\cite{Viterbo3}, appendix)]\label{lemma:5.2}
Let $v:\Sigma\rightarrow M$ be a connected minimal surface passing through $x\in M$, and such that $v(\partial\Sigma)\cap B(x,r)=\emptyset$ with $r\leq 1/C$. Then, \[\area\left[v(\Sigma)\right]\geq \pi r^2\exp\left[\varphi\left(Cr\right)\right]\, ,\] where the function $\varphi$ is defined below.
\end{lemma}
\begin{proof} For the sake of completeness, we recall the proof from the appendix of \cite{Viterbo3}. For $s\leq r$, set $a(s)=\area\left[v(\Sigma)\cap B(x,s)\right]$. As $v$ is a minimal surface, $a(s)$ must be less than or equal to the area of the cone $C(s)$ spanned by the (possibly singular) curve $C\cap \partial B(x,s)$ (see figure \ref{Figure:Cone}). Let us estimate the area of $C(s)$ for a value of $s$ such that the curve $\partial C(s)$ is not singular. \begin{align*}
\area\left[C(s)\right] & =\int_0^s L\left[\exp_{x}\left(\frac{s'}{s}\exp^{-1}_{x}\left(\partial C(s)\right)\right)\right]ds'\\  & \leq \int_0^s  \frac{s'}{s}\frac{\sinh(Cs)}{\sin(Cs)}L \left[ \partial  C(s)\right]ds'\\  & = \frac{s}{2}\frac{\sinh(Cs)}{\sin(Cs)}L \left[ \partial  C(s)\right]\, ,
\end{align*}
where the inequality directly follows from the comparison theorems in Riemannian geometry. Hence, using the minimality of $v$, we have: \begin{align*} a(s)\leq \area\left[C(s)\right] & \leq \frac{s a'(s)}{2} \frac{\sinh(Cs)}{\sin(Cs)} & \mbox{ or equivalently, } \frac{a'(s)}{a(s)} & \geq  \frac{2}{s}\frac{\sin(Cs) }{\sinh(Cs)}\, .\end{align*} Set \[\varphi(s)=\int_0^s \frac{2}{s'}\left[\frac{\sin(s')}{\sinh(s')}-1\right]ds'\, .\] Recall that $a(\epsilon)\geq \pi\epsilon^2$. By integrating the above inequality, we get: \[\log\left(\frac{a(r)}{\pi r^2}\right)\geq {\left[\log\left(\frac{a(s')}{s'^2}\right)\right]}_0^r\geq \varphi\left(Cr\right)\, .\] \end{proof}
\begin{figure}
\centering
\includegraphics[height=50mm]{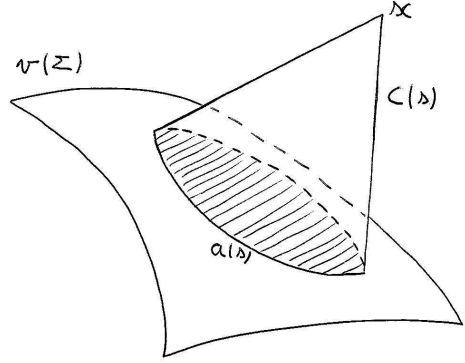}
\caption{Cone $C(s)$.}\label{Figure:Cone}
\end{figure}
\begin{proof}[Proof of proposition \ref{prop:BoundsEnergy}.] Assume $d=\diam\left[u(\Sigma)\right]> (2N)/C$. Then, there exist $N$ points $z_1,\dots , z_N$ such that the balls $B\left(u(z_i),1/C\right)$ are pairwise disjoint. Those points may be chosen so that $u(\partial\Sigma)$ lies outside the corresponding balls. By applying lemmas \ref{lemma:5.1} and \ref{lemma:5.2}, we easily get:
\[E(u)\geq \frac{1}{A}\area(u)\geq N\frac{\pi \exp\left(\varphi(1)\right)}{A C^2}\, .\] Say $N+1$ to be the upper integer part of $Cd/2$. Then, we have:
\[E(u)\geq \left(\left\lceil \frac{Cd}{2}\right\rceil-1\right)
\frac{\pi\exp\left(\varphi(1)\right)}{AC^2}\, ,\] which concludes the proof.
\end{proof}
\subsubsection{\!\!\!\!} Without proof, Recall:
\begin{lemma}[Hofer-Salamon \cite{HoferSalamon}]\label{Lemma:HoferSalamon} Set $\BBB_T=[-T,T]\times [0,1]$. For a $\BBB_T$-parametrized family of almost complex structures $\JJJ=\{J_{s,t}\}$ in $\Icc_A(J_0)$, there exist~$\delta>0$ and $\eta>0$ (depending on~$A$), such that the following holds. For each $\JJJ$-holomorphic curve $u:(\BBB_T,\partial\BBB_T) \rightarrow (M,L)$ with~$E(u)<\delta$, we have: \[\forall z,z'\in \BBB_S,\, d\left(u(z),u(z')\right)\leq \exp\left[\eta\left(S-T\right)\right]\sqrt{E(u)}\, .\]
\end{lemma}
\subsubsection{\!\!\!\!} The following lemma compares the energy of Floer continuation strips and the difference of actions, as defined in subsection \ref{Subsection:FilteredLagrangianFloerHomology}. Here, just recall
\[\Acc_{H^-}(x_-)-\Acc_{H^+}(x_+)=\int_{\BBB} u^*\omega+\int_0^1H_t(x_-(t))dt-\int_0^1H_t(x_+(t))dt\, .\]
\begin{lemma}\label{lemma:A3}
Let $u:(\BBB,\partial\BBB)\rightarrow (M,L)$ be a $(\HHH,\JJJ)$-Floer continuation strip from $x_-$ to $x_+$. Then
\[\alpha_-(\HHH)\leq E(u)-\left[\Acc_{H_-}(x_-)-\Acc_{H_+}(x_+)\right]\leq \alpha_+(\HHH)\, ,\]
where \begin{align}
\alpha_+(\HHH) & =\int_{-\infty}^{\infty}\int_0^1\left(\sup\partial_s H_{s,t}\right) \, dtds\, ,\\
\mbox{and } \alpha_-(\HHH) & = \int_{-\infty}^{\infty} \int_0^1\left(\inf\partial_s H_{s,t}\right)\, dt ds\, .\label{Eqn:EstimateEnergies}
\end{align}
\end{lemma}
\begin{proof} This follows from the following computations.
\begin{align*}
E(u) & = \int_{-\infty}^{+\infty}\int_0^1\omega(\partial_su,J\partial_su)dtds\\
 & = \int_{-\infty}^{+\infty}\int_0^1\omega(\partial_su,\partial_t u-X_{s,t}(u))dtds\\
 & = \int_{\BBB}u^*\omega-\int_{-\infty}^{+\infty}\int_0^1dH_{s,t}(\partial_su)dtds\\
 & = \left[\Acc_{H_-}(x_-)-\Acc_{H_+}(x_+)\right]+\int_{-\infty}^{+\infty}\int_0^1\partial_sH_{s,t}(u)\, dt ds\, .
\end{align*}
We have done. \end{proof}
\subsection{Remarks on the original proof due to Chekhanov.}\label{Subsection:RemarksOriginalProofChekhanov}
Originally, the proof of Chekanov is slightly different from the one presented in section \ref{Section:Chekhanov'sProof}. Let us explain why the approches are exactly the same. We introduce the auxiliary Hamiltonian:
\[F(t,x)=-H\left(t,\varphi_t^H (x)\right)\]
which satisfies the same conditions than $H$. Recall $\varphi_t^H\varphi_t^F=\id$ (\cite{HoferZehnder}, proposition 1, p. 144). In~\cite{Chek98}, Chekanov considered the displacement of $L$ as a whole by setting $L_s=\varphi_s^FL$. Set \[\Pi=\left\{(s,\xi),\, \xi\in W^{1,2}\left([0,1],M\right),\, \xi(0)\in L_0,\xi(1)\in L_s \right\}\] Chekanov defined a $\RRR/2a \ZZZ$-valued functional on $\Pi$ as a primitive of the one-form $\alpha-\beta$, where \begin{align*} \alpha(s,\xi) & = \int_0^1\omega(\dot{\xi},\cdot)\, , & \mbox{
and } \beta(s,\xi) & = F(s,\gamma(1))ds\, .
\end{align*}
Let $\Omega$ be the space of $L$-connecting paths of class $W^{1,2}$. The following map is a diffeomorphism
\[\Phi:\begin{matrix}
[0,1]\times \Omega & \longrightarrow & \Pi\\
(s, \gamma) & \longmapsto & \xi:t\mapsto \varphi^F_{st}(\xi_t)\, ,
\end{matrix}\]
whose differential is
\begin{align*}
d\Phi(s,\gamma)(\delta s,\delta \gamma) & = \left(\delta s, t\delta s Y_{st}(\gamma_t)+d\varphi^F_{st}(\gamma_t)\delta \gamma_t\right)\, .
\end{align*}
\begin{lemma}
With the above notations, we have:
\begin{align*}
\Phi^*\alpha-\Phi^*\beta & =d\Acc\\ \mbox{where } \Acc(s,\gamma) & = \int_0^1 sH_{st}\left[\gamma(t)\right]dt -\oint_x \omega\, .
\end{align*}
\end{lemma}
\noindent For $s$ fixed, note that $\{sH_{st}\}$ is the Hamiltonian which generates the Hamiltonian path $\{\varphi^H_{st}\}$.
\begin{proof} As a preliminary, we compute the following derivation: \begin{align*}
\partial_t\left[t H_{st}(\gamma_t)\right] & = H_{st}\left(\gamma_t \right)+ t dH_{st}\left(\dot{\gamma}_t\right) + ts \left((\partial_t H)_{st}\right) \left(\gamma_t \right) \\  & = t dH_{st}(\dot{\gamma}_t) +\partial_s\left[s H_{st}\left( \gamma_t  \right)\right]\, . \end{align*}
By integrating from $0$ to $1$, it comes:
\begin{align*}
\int_0^1\partial_s\left[s H_{st}\left( \gamma_t  \right)\right] & = +H\left(s,\gamma(1)\right)- \int_0^1  t dH_{st}\left(\dot{\gamma}_t\right)\, .
\end{align*}
Here is the computations of $\Phi^*\alpha$ and $\Phi^*\beta$:
\begin{align*}
(\Phi^*\alpha)(s,\gamma)(\delta s,\delta \gamma) & = \int_0^1 \omega\left(s Y_{st}(\gamma_t) +d\varphi^F_{st} \dot{\gamma}_t,t\delta s Y_{st}(\gamma_t)+ d\varphi^F_{st}(\gamma_t)\delta\gamma_t\right)\\  & = \int_0^1 \left[ -s d\left(F_{st}\circ \varphi^F_{st}\right)  \delta\gamma_t + t\delta s d\left(F_{st}\circ \varphi^F_{st}\right) \partial_t\gamma_t+  \omega\left(\dot{\gamma}_t,\delta\gamma_t\right)\right]\\  & = \int_0^1 \left[\omega(\dot{\gamma}_t,\delta\gamma_t)+s
 dH_{st}\left(\delta\gamma_t\right)\right]-\delta s \int_0^1 t  dH_{st}\left[\dot{\gamma}_t\right]\, ,\\ \mbox{and }(\Phi^*\beta)(s,\gamma)(\delta s,\delta \gamma) & = -H\left(s,\gamma(1)\right) \delta s\, .
\end{align*}
Hence, we get:
\begin{align*}
\left[\Phi^*\alpha-\Phi^*\beta\right](s,\gamma)(\delta s,\delta \gamma) & = \int_0^1\left[\omega\left( \dot{\gamma}_t,\delta\gamma_t\right)+s dH_{st}\cdot \delta \gamma_t\right]+\delta s \int_0^1 \partial_s\left[s H_{st}(\gamma_t) \right]\\ & = d\Acc(s,\gamma)(\delta s,\delta \gamma)\, .
\end{align*}
\end{proof}
Once this computation made, it is clear that the two approaches are equivalent. By completing the $L$-orbits by half-disks, we simply take into account possible disks with zero
Maslov index in the kernel of the symplectic action $\omega:\pi_2(M,L)\rightarrow \RRR$, which gives a graduation.

Another slight difference, Chekanov defined the continuation morphisms from $\epsilon H_{\epsilon t}$ to $H_t$ and from $H_t$ to $\epsilon H_{\epsilon t}$. 
\newpage
\small{
}
\end{document}